\documentclass{amsart}
\usepackage{tikz}
\usepackage{xcolor}
\usepackage{amssymb,latexsym,amsmath,extarrows}
\usepackage{graphicx,mathrsfs,comment}
\usepackage{hyperref,url}

\usepackage{amstext}
\usepackage{bbm}

\usepackage{varwidth}

\usepackage{graphicx}
\usepackage{overpic}
\usepackage{amsmath}
\usepackage{enumitem}
\usepackage{hyperref}
\usetikzlibrary{shapes.geometric}
\usetikzlibrary{3d,calc}
\usepackage{amssymb}
\usepackage{mathtools}
\usepackage{epstopdf,amsmath}
\usepackage{amsthm}
\usepackage{xcolor}
\usepackage{upref}
\numberwithin{equation}{section}

\numberwithin{equation}{section}

\newtheorem{theorem}{Theorem}[section]
\newtheorem{lemma}[theorem]{Lemma}

\newtheorem{remark}[theorem]{Remark}

\newtheorem{definition}[theorem]{Definition}

\newcommand{\al}{\alpha}

\newcommand{\ga}{\gamma}
\newcommand{\Ga}{\Gamma}
\newcommand{\de}{\delta}

\newcommand{\De}{\Delta}
\newcommand{\e}{\varepsilon}

\newcommand{\Th}{\Theta}

\newcommand{\si}{\sigma}
\newcommand{\Si}{\Sigma}
\newcommand{\vp}{\varphi}
\newcommand{\om}{\omega}
\newcommand{\Om}{\Omega}

\newcommand{\cj}{\mathcal J}

\newcommand{\cn}{\mathcal N}

\newcommand{\ck}{\mathcal K}

\newcommand{\wt}{\widetilde}
\newcommand{\wh}{\widehat}

\newcommand{\ZR}{\mathbb{R}}

\newcommand{\ZZ}{\mathbb{Z}}

\newcommand{\ZS}{\mathbb{S}}

\newcommand{\Id}{{\bf{1}}}

\newcommand{\bA}{{\bf A}}

\newcommand{\Tau}{\mathcal{T}}

\newcommand{\R}{\mathbb{R}}

\newcommand{\vh}{\boldsymbol h}

\newcommand{\dist}{{\rm dist}}

\newcommand{\supp}{{\rm{supp}}}

\newcommand{\bth}{\boldsymbol \Theta}
\newcommand{\vth}{\vartheta}

\begin{document}

\title{Square function estimates for conical regions}

\date{}

\begin{abstract}
    We prove square function estimates for certain conical regions. Specifically, let $\{\De_j\}$ be regions of the unit sphere $\ZS^{n-1}$ and let $S_jf$ be the smooth Fourier restriction of $f$ to the conical region $\{\xi\in\R^n:\xi/|\xi|\in\De_j\}$. We are interested in the following estimate 
    $$\Big\|(\sum_j|S_jf|^2)^{1/2}\Big\|_p\lesssim_\e \de^{-\e}\|f\|_p.$$ 
    The first result is: when $\{\De_j\}$ is a set of disjoint $\de$-balls, then the estimate holds for $p=4$. The second result is: In $\R^3$, when $\{\De_j\}$ is a set of disjoint $\de\times\de^{1/2}$-rectangles contained in the band $\ZS^2\cap N_\de(\{\xi_1^2+\xi_2^2=\xi_3^2\})$ and $\supp\wh f\subset \{\xi\in\R^3:\xi/|\xi|\in\ZS^2\cap N_\de(\{\xi_1^2+\xi_2^2=\xi_3^2\})\}$, then the estimate holds for $p=8$. The two estimates are sharp.

\end{abstract}

\author{Shengwen Gan} \address{Shengwen Gan\\  Deparment of Mathematics, Massachusetts Institute of Technology, USA}\email{shengwen@mit.edu}

\author{Shukun Wu} \address{ Shukun Wu\\  Department of Mathematics\\ California Institute of Technology, USA}\email{skwu@caltech.edu}

\maketitle


\section{Introduction}

The whole Littlewood-Paley theory concerns orthogonality properties of the Fourier transform, and square function gives a way to express
and quantify orthogonality of the Fourier transform on $L^p$ space. In particular, one seeks estimates of the form
\begin{equation}
\label{general-LP}
    \Big\|\Big(\sum_{\si\in\Si}|S_\si f|^2\Big)^{1/2}\Big\|_p\leq C_{n,p,\Si}\|f\|_p,
\end{equation}
where $\Si=\{\si\}$ is a collection of geometric objects in $\ZR^n$ and $S_\si f$ is the Fourier restriction of $f$ to $\si$. 
Two different types of the operator are frequently studied: smooth operator and sharp operator. If one studies the smooth operator, then one uses the definition $S_\si f=(\Id^*_\si\wh f\ )^\vee$, where $\Id^*_\si$ is a smooth bump function at $\si$. Similarly, one uses the definition $S_\si f=(\Id_\si \wh f\ )^\vee$ where $\Id_\si$ is the indicator function of $\si$ to study the sharp operator. Usually, the smooth version is easier than the sharp version.

Let us discuss some well-known square function estimates of the form \eqref{general-LP}.
The classic Littlewood-Paley theory justifies \eqref{general-LP} for $1<p<\infty$ when $\Si$ is the collection of dyadic annuli and $\{S_\si\}_{\si\in\Si}$ are sharp (or smooth) Fourier projection operators associated to the annuli; the Rubio de Francia’s square function estimate justifies \eqref{general-LP} for $2\leq p<\infty$ when $\Si$ is a collection of disjoint rectangles whose edges are parallel to the coordinate axes and $\{S_\si\}_{\si\in\Si}$ are sharp (or smooth) projections (see \cite{Rubio}, \cite{Jean-Lin} and \cite{Lacey-LP}). If one uses the sharp operator and seeks for an square function estimate, then the shape of each $\si\in\Si$ is quite limited. Indeed, due to Fefferman's ball-multiplier example \cite{fefferman1971multiplier}, we see that \eqref{general-LP} makes sense only when $\si$ is ``flat" in some sense.
In the case of sharp operator, we replace each $\si\in\Si$ by a rectangular box of the same size to make sure the boundaries of $\si$ are flat.
For this specific choice of $\Si$, the estimate \eqref{general-LP} is related to the maximal Nikodym or Kakeya conjecture (see for instance \cite{Cordoba-BR}, \cite{Bourgain-Besicovitch}), which is one of the core conjectures in harmonic analysis. In this paper, we study \eqref{general-LP} when $\Si$ is a certain collection of conical regions with the common apex at the origin.

\subsection{First result}
A well-known result for square function in $\R^2$ proved by C\'ordoba \cite{Cordoba} says that for some absolute constant $C$ the following estimate holds: 
\begin{equation}\label{cordobaest}
     \Big\|\big(\sum_{j=1}^N|S_j f|^2\big)^{1/2}\Big\|_4\lesssim (\log N)^C\|f\|_4.
\end{equation}
Here $S_jf$ is the Fourier restriction to the conical regions determined by
\begin{equation}
\label{planar-triangulation}
    \De_j=\{\om\in\ZS^1:2\pi j/N\leq \arg(\om)\leq 2\pi(j+1)/N\},
\end{equation}
namely, $\wh{S_j f}(\xi)=\Id_{\De_j}(\xi/|\xi|)\wh{f}(\xi)$.

Our first result is a generalization of Cordoba's estimate \eqref{cordobaest} to higher dimensions.
In Cordoba's square function estimate, the regions are chosen as $\Si=\{\si_j\}$ where $\si_j=\{\xi: \xi/ |\xi|\in\De_j\}$ is a sector of angle $N^{-1}$. 
A natural way to generalize it to higher dimensions is by choosing $\{\De_j\}$ to be a set of disjoint $\de$-balls in $\ZS^{n-1}$.
We first discuss the smooth version. Let $\{\De_j\}$ be a set of disjoint $\de$-balls
and let $\{\Id^*_{\De_j}\}$ be the corresponding smooth bump functions. More precisely,
$\Id^*_{\De_j}$ is supported in $\De_j$ and $=1$ on $\frac{1}{2}\De_j$ ($\frac{1}{2}\De_j$ is the ball that has the same center as $\De_j$ but half the radius), and it satisfies the decay condition $|\nabla^k \Id^*_{\De_j}(\om)|\lesssim_k \de^{-k}$.

\begin{theorem}
\label{main-thm}
Let $\{\De_j\}$ be a set of disjoint $\de$-balls in $\ZS^{n-1}$ and $\{\Id^*_{\De_j}\}$ be the corresponding smooth bump functions.
Let $T_jf=(\Id_{\De_j}^\ast(\cdot/|\cdot|)\wh{f}(\cdot))^\vee$ be the smooth Fourier projection of $f$ to the conical region $\{\xi\in\R^n:\xi/|\xi|\in\De_j\}$. Then for some constant $C$ depending only on the dimension, we have
\begin{equation}
\label{main-esti}
     \Big\|\big(\sum_{j}|T_j f|^2\big)^{1/2}\Big\|_4\lesssim (\log \de^{-1})^{1/2}\|f\|_4.
\end{equation}
\end{theorem}

\begin{remark}
{\rm
The exponent $p=4$ is sharp in the sense that for $p>4$, the factor $(\log \de^{-1})^{1/2}$ in \eqref{main-esti}
should be replaced by some factor like $\de^{-c_p}$, which depends exponentially on $\de$. Also, a logarithmic loss like $(\log \de^{-1})^{1/4}$ is inevitable due to the existence of Besicovitch set (see for instance \cite{accomazzo2020directional} Section 8.6).}
\end{remark}

\medskip

Next, let us discuss the sharp-projection version of Theorem \ref{main-thm}.
We still hope each $\De_j$ is roughly a $\de$-ball in $\ZS^{n-1}$. While due to Fefferman's ball-multiplier example, the boundaries of $\si_j=\{\xi\in\R^n:\xi/|\xi|\in\De_j\}$ must be flat, otherwise estimates like \eqref{main-esti} for sharp projection could fail. These two conditions suggest that each $\De_j$ is a ``$\de$-regular polyhedron" in $\ZS^{n-1}$, for which we give the precise definition below.

\begin{definition}
Let $\De$ be a subset of $\ZS^{n-1}$ and $0<\de\ll1$. We say that $\De$ is a $\de$-regular polyhedron of $\ZS^{n-1}$ if $\De$ is surrounded by great $(n-2)$-spheres of $\ZS^{n-1}$ and there exists $\om_\De\in \ZS^{n-1}$ such that $\ZS^{n-1}\cap B_{c\de}(\om_\De)\subset \De\subset \ZS^{n-1}\cap B_{C\de}(\om_\De)$, where $c\ll C$ are two absolute constants. We call the portion of the great $(n-2)$-sphere that form the boundary of $\De$ the face of $\De$. Throughout the paper, we assume the polyhedron $\De$ has $O(1)$ many faces.
\end{definition}

We are interested in the set of disjoint $\de$-regular polyhedrons. Let us discuss some examples here.
The collection $\{\De_j\}$ given by \eqref{planar-triangulation} is a set of disjoint $N^{-1}$-regular polyhedrons in $\ZS^1$. In higher dimensions, we can also easily choose a triangulation of the sphere: $\ZS^{n-1}=\sqcup_j \De_j $ such that each $\De_j$ is a $\de$-regular polyhedron ($(n-1)$-simplex) of $\ZS^{n-1}$.

However, it turns out that we still need another condition for the collection $\{\De_j\}_j$. In fact, how the normal directions of the faces of each $\De_j$ are distributed is critical. Recall that each face of $\De_j$ is a portion of a great $(n-2)$-sphere, and the normal direction of the face is the normal direction of the corresponding great $(n-2)$-sphere in $\ZS^{n-1}$. Denote by $\mathcal N$ the collection of all the unit normal directions of all the faces of all the $\De_j$. One can think of $\mathcal N$ as a subset of $\ZS^{n-1}$. 

\begin{definition}\label{1dim}
We call a collection of $\de$-regular polyhedrons $\{\De_j\}$ ``one-dimensional", if $\mathcal N$ is contained in $O(1)$ many great circles  in $\ZS^{n-1}$.
\end{definition} 
For instance, the $\{\De_j\}$ given in \eqref{planar-triangulation} is one-dimensional.
Another example is the ``pyramid" example. Consider the set $[-1,1]^{n-1}\times \{-1\}$ in $\R^n$. We partition it into the sets of form $D_{\vec b}=\prod_{j=1}^{n-1}[\frac{b_j}{N},\frac{b_j+1}{N}]\times \{-1\} $, where $\vec b=(b_1\cdots,b_{n-1})$ and $-N\le b_j\le N-1$ are integers.
For each $\vec b= (b_1,\cdots b_{n-1})$, we define the small pyramid $\si_{\vec b}$ which consists of the points $\xi$ such that the ray $\overrightarrow{0 \xi}$ emanating from the origin intersects $D_{\vec b}$. More precisely,

$$ \si_{\vec b}:=\{\xi\in\R^n: \xi_n<0, \frac{1}{|\xi_n|} \xi\in D_{\vec b} \}. $$
We set $\De_{\vec b}=\ZS^{n-1}\cap \si_{\vec b}$. Then, $\{\De_{\vec b}\}_{\vec b}$ is one dimensional.

We can state our result for the sharp operator. 
\begin{theorem}
\label{main-thm-2}
Suppose $\{\De_j\}_j$ is a set of disjoint $\de$-regular polyhedrons and is one-dimensional. Define $S_jf=(\Id_{\De_j}(\cdot/|\cdot|)\wh{f}(\cdot))^\vee$ which is the Fourier restriction of $f$ to $\{\xi\in\R^n:\xi/|\xi|\in\De_j\}$.  Then for any $\e>0$ we have
\begin{equation}
\label{main-esti-2}
     \Big\|\big(\sum_{j}|S_j f|^2\big)^{1/2}\Big\|_4\lesssim_\e \de^{-\e} \|f\|_4.
\end{equation} 
\end{theorem}
\noindent The one-dimensional assumption on $\{\De_j\}_j$ is somewhat necessary. In fact, we will show in the Appendix that without the ``one-dimensional" condition,
then \eqref{main-esti-2} can fail.

\begin{remark}
\rm
One possible application of Theorem \ref{main-thm} as well as Theorem \ref{main-thm-2} is the study of the local smoothing conjecture. Actually, the reverse square function estimate for cone plus the Nikodym maximal estimate for cone together with Theorem \ref{main-thm} would imply the local smoothing conjecture for cone. 
See for instance \cite{Tao-Vargas-LS}.

\end{remark}

\begin{remark}
\rm
After this project was finished, the authors became aware that Francesco Di Plinio and Ioannis Parissis have studied the same problem as in Theorem \ref{main-thm}. Their result  (\cite{di2021maximal} Theorem J) will imply our inequality \eqref{main-esti} (with the same constant ($\log\de^{-1})^{1/2}$) by a standard interpolation argument (see for example in \cite{demeter2010singular} Lemma 3.1 or \cite{Katz-two-d-directional} Proposition 2.1). Their method is based on the time-frequency analysis, while our method is based on the high-low method developed recently (see \cite{guth2020sharp}, \cite{guth2019incidence}, \cite{guth2020improved}). 
We also remark that when $n=2$, both the smooth operator version and the sharp operator version were studied by Accomazzo, Di Plinio, Hagelstein, Parissis and Roncal \cite{accomazzo2020directional}.
\end{remark}

\subsection{Second result}
Let us talk about the second result.
We consider the cone $\Ga=\{ \xi\in\R^3:\xi_1^2+\xi_2^2=\xi_3^2, 1/2\le |\xi_3|\le 1 \}$, and let $N_\de (\Ga)$ denote its $\de$-neighborhood. There is a canonical covering of $N_\de(\Ga)$ using finitely overlapping planks $\tau$ of dimensions $\sim \de\times \de^{1/2}\times 1$. Denote this collection by $\Tau=\{\tau\}$. For each $\tau\in\Tau$, choose a smooth bump function $\Id^*_\tau$ adapted to $\tau$ so that $\supp \Id^*_\tau\subset 2\tau$ and $\Id^*_{\tau}=1$ on $\tau$. Define $f_\tau:=(\Id^*_\tau \wh f\ )^\vee$ as usual.
Guth, Wang and Zhang \cite{guth2020sharp}  proved the following sharp $L^4$ reverse square function estimate.

\begin{theorem}[Reverse square function estimate, \cite{guth2020sharp}]
Assuming $\supp\wh f\subset N_\de(\Ga)$, then for any $\e>0$ we have
\begin{equation}
    \|f\|_{L^4(\R^3)}\le C_\e \de^{-\e} \bigg\| (\sum_\tau|f_\tau|^2)^{1/2} \bigg\|_{L^4(\R^3)}.
\end{equation}
Or equivalently,
\begin{equation}\label{reverse}
     \|\sum_{\tau}f_\tau\|_{L^4(\R^3)}\le C_\e \de^{-\e} \bigg\| (\sum_\tau|f_\tau|^2)^{1/2} \bigg\|_{L^4(\R^3)}.
\end{equation}
\end{theorem}

In this paper, we prove the sharp $L^8$ square function estimate for the cone.

\begin{theorem}\label{L8thm}
Assuming $\supp\wh f\subset N_\de(\Ga)$, then for any $\e>0$ we have
\begin{equation}\label{L8}
     \bigg\| (\sum_\tau|f_\tau|^2)^{1/2} \bigg\|_{L^8(\R^3)}\le C_\e \de^{-\e} \|f\|_{L^8(\R^3)}.
\end{equation}
\end{theorem}

\begin{remark}
\rm

The endpoint $p=8$ is sharp in the sense that if we consider the estimate
$$ \bigg\| (\sum_\tau|f_\tau|^2)^{1/2} \bigg\|_{L^p(\R^3)}\le C(p,\de) \|f\|_{L^p(\R^3)} $$ 
for $p>8$, then the best constant $C(p,\de)$ should be a positive power of $\de^{-1}$. We will give a sharp example in the Appendix.

If we remove the condition $\supp \wh f\subset N_{\de}(\Ga)$, then the best we can hope is an $L^4$-estimate. This result was implicitly proved by Mockenhaupt, Seeger and Sogge \cite{mockenhaupt1992wave}. Actually, by a duality argument, one can reduce the $L^4$ square function estimate to a maximal Nikodym estimate which is Lemma 1.4 in \cite{mockenhaupt1992wave}. 

From Theorem \ref{main-thm} in this paper, we easily see \eqref{L8} holds for $p=4$. Moreover, if we use trilinear restriction estimate for the cone, then we can prove \eqref{L8} for $p=6$. In order to prove for $p=8$, we need to do more work.
\end{remark}

\medskip

By Littlewood-Paley theory, we can prove a global version of Theorem \ref{L8thm}. Let us consider $\mathbb B^2:=\ZS^2\cap N_{\de}(\Ga)$ which is a band of width $\de$ in $\ZS^2$. Let $\{\De_j\}$ be a set of disjoint $\de\times\de^{1/2}$-rectangles that are contained in $\mathbb B^2$. Also, we choose $\{\Id^*_{\De_j}\}$ to be the corresponding smooth bump functions. We see that $\Id^*_{\De_j}(\xi/|\xi|)$ is a smooth cutoff function in the region $\{\xi\in\R^3: \xi/|\xi|\in\De_j\}$. Define $T_j f:=(\Id^*_{\De_j}(\xi/|\xi|)\wh f(\xi))^\vee$. Our global version is the following.

\begin{theorem}\label{globalL8}
Assuming $\wh f\subset \{\xi\in\R^3: \xi/|\xi|\in\mathbb B^2\}$, then for any $\e>0$ we have
\begin{equation}
    \big\| (\sum_j|T_j f|^2)^{1/2} \big\|_{L^8(\R^3)}\le C_\e \de^{-\e}\|f\|_{L^8(\R^3)}.
\end{equation}
\end{theorem}

In the end of this section, let us talk about the structure of the paper. In Section \ref{sec2}, we do a global-to-local reduction in the frequency space to reduce the problem to the case that $\supp\wh f\subset \{|\xi|\sim 1\}$. In Section \ref{sec3}, we prove Theorem \ref{main-thm}. In Section \ref{sec4}, we prove Theorem \ref{main-thm-2}. In Section \ref{sec5}, we prove Theorem \ref{L8thm}. In the Appendix, we discuss some examples. 

\medskip

\noindent
{\bf Acknowlegement}. The authors would like to thank Francesco Di Plinio and Ioannis Parissis for some useful discussions, and for bringing their papers to our attention.

\section{The global-to-local reduction}\label{sec2}
The main goal of this section is to reduce Theorem \ref{main-thm} and Theorem \ref{main-thm-2} to a local version. That is to say, we only need to prove the case when $\supp\wh f\subset \{|\xi|\sim 1\}$.

We may assume $\De_j$ are within the conical region $\{\om\in\ZS^{n-1}:\om\cdot e_n\ge 1\}$ and $\{\De_j\}$ are $100C\de$-separated. Let $P_k$ be the Littlewood-Paley operator for the dyadic annulus, so that for any function $f$, $\wh{P_kf}$ is supported in the annulus $\{\xi:|\xi|\sim 2^k\}$ and $f=\sum_{k\in\ZZ}P_k f$. 
More precisely, we choose a function $\rho(r)$ supported in $[1/4,4]$ such that $\sum_{k\in\ZZ}\rho(r/2^k)=1$, and set $P_k f:=(\rho(|\xi|/2^k)\wh f(\xi))^{\vee}$. Choose $m=10^n$ which is a large number. 
For each integer $0\le i\le m-1$,
define $\ck_i:=m\ZZ+i$, so that there is a partition $\ZZ=\sqcup_{0\le i\le m-1} \ck_i$. 
We have 
\begin{align}
\label{before-LP}
\Big\|\Big(\sum_{j}|T_j f|^2\Big)^{1/2}\Big\|_4^4=
    &\Big\|\Big(\sum_{j}\Big|\sum_{k
    \in \ZZ}T_j P_kf\Big|^2\Big)^{1/2}\Big\|_4^4\\
\lesssim \sum_{0\le i\le m-1}
    &\Big\|\Big(\sum_{j}\Big|\sum_{k
    \in \ck_i}T_j P_kf\Big|^2\Big)^{1/2}\Big\|_4^4.
\end{align}
For simplicity, we just write $\ck_i$ as $\ck$. The only property of $\ck$ we will use is that for any two different numbers $k_1, k_2\in\ck$, we have $|k_1-k_2|\ge 10^n$.
The goal of Theorem \ref{main-thm} is to prove
\begin{equation}
    \Big\|\Big(\sum_{j}\Big|\sum_{k
    \in \ck}T_j P_kf\Big|^2\Big)^{1/2}\Big\|_4^4\lesssim (\log\de^{-1})^{2}\|f\|_4^4.
\end{equation}
For convenience, let us denote
\begin{equation}
    I=\Big\|\Big(\sum_{j}\Big|\sum_{k
    \in \ck}T_j P_kf\Big|^2\Big)^{1/2}\Big\|_4^4, \hspace{.5cm} II=\Big\|\Big(\sum_{j}\sum_{k
    \in \ck}|T_j P_kf|^2\Big)^{1/2}\Big\|_4^4.
\end{equation}
We begin with the first estimate.
\begin{lemma}
Let I and II be as above. We have
$$ I\le 1000 II. $$
\end{lemma}

\begin{proof}
After expanding the $L^4$-norm, we get
\begin{align}
\label{off-diagonal-1}
    I&= \sum_{j_1\not=j_2}\sum_{k_1,k_2,k_3,k_4\in\ck}\int (T_{j_1} P_{k_1}f\cdot T_{j_2} P_{k_2}f)(\overline{T_{j_1} P_{k_3}f\cdot T_{j_2} P_{k_4}f})\\ \label{diagonal-1}
    &+\sum_{j_1=j_2}\sum_{k_1,k_2,k_3,k_4\in\ck}\int (T_{j_1} P_{k_1}f\cdot T_{j_2} P_{k_2}f)(\overline{T_{j_1} P_{k_3}f\cdot T_{j_2} P_{k_4}f}).
\end{align}
Denote $\vec k=(k_1,k_2,k_3,k_4)$. In the following discussion, we would like to find a partition of the set $\ck^4$ and hence a partition of the summation $\sum_{\vec k\in\ck^4}$.
We will discuss the two terms \eqref{off-diagonal-1}, \eqref{diagonal-1} separately.

\smallskip

\textit{Case 1}: $j_1\not=j_2$.

When $j_1\not=j_2$, for the quadruples $\vec k=(k_1,k_2,k_3,k_4)\in\ck^4$, we consider the following subsets of $\ck^4$:
\begin{enumerate}
    \item $A=\{k_1=k_3\}$.
    \item $B=\{k_2=k_4\}$.
    \item $C=\{k_1=k_3,k_2=k_4\}$.
    \item $D=\{k_1\not=k_3,k_2\not=k_4\}$.
\end{enumerate}
Then we see that $\Id_{\ck^4}=\Id_A+ \Id_B+ \Id_D- \Id_C$. 

For $\vec k\in D$, we have by Plancherel that
\begin{align}
    &\int (T_{j_1} P_{k_1}f\cdot T_{j_2} P_{k_2}f)(\overline{T_{j_1} P_{k_3}f\cdot T_{j_2} P_{k_4}f})\\ \label{L4}
    =&\int \wh{T_{j_1} P_{k_1}f}\ast \wh{T_{j_2} P_{k_2}f}\cdot\overline{\wh{T_{j_1} P_{k_3}f}\ast\wh{T_{j_2} P_{k_4}f}}.
\end{align}
We claim this integral is $0$. First, by definition we have \begin{equation}
    \supp \wh{T_{j} P_{k}f}\subset \{\xi: \xi/|\xi|\in \De_j, |\xi|\sim 2^k \}.
\end{equation}
Define
\begin{equation}
    \tau_{k,j}:=\{\xi: \xi/|\xi|\in \De_j, |\xi|\sim 2^k \},
\end{equation}
then we see that $\tau_{k,j}$ is morally a tube of length $2^k$ and radius $\de 2^k$, pointing to the direction $c_{\De_j}\in \ZS^{n-1}$ ($c_{\De_j}$ is the center of $\De_j$).
We note that the support of the integrand of \eqref{L4} is contained in $(\tau_{j_1,k_1}+\tau_{j_2,k_2})\cap (\tau_{j_1,k_3}+\tau_{j_2,k_4})$. To show $\eqref{L4}=0$, it suffices to show $(\tau_{j_1,k_1}+\tau_{j_2,k_2})\cap (\tau_{j_1,k_3}+\tau_{j_2,k_4})=\emptyset$, or equivalently $$(\tau_{j_1,k_1}-\tau_{j_1,k_3})\cap (\tau_{j_2,k_4}-\tau_{j_2,k_2})=\emptyset.$$
Since $\vec k\in D$, we have $k_1\neq k_3$, so $|k_1-k_3|\ge 10^n >>1$. We may assume $k_1>k_3$. By an easy geometry, we see that $\tau_{j_1,k_1}-\tau_{j_1,k_3}\subset (1+\frac{1}{10})\tau_{j_1,k_1}$. Here $(1+\frac{1}{10})\tau_{j_1,k_1}$ is the dilation of $\tau_{j_1,k_1}$ with respect to its center (regarding $\tau_{j_1,k_1}$ as a tube). Similarly, we have $\tau_{j_2,k_2}-\tau_{j_2,k_4}\subset (1+\frac{1}{10})\tau_{j_2,\max(k_2,k_4)}$. Note that $(1+\frac{1}{10})\tau_{j_1,k_1}$ is contained in the conical region $\{\xi: \xi/|\xi|\in N_{10C\de}\De_{j_1}\}$, while $(1+\frac{1}{10})\tau_{j_2,\max(k_2,k_4)}$ is contained in the conical region $\{\xi: \xi/|\xi|\in N_{10C\de}\De_{j_2}\}$. Since $j_1\neq j_2$, we have $\dist(\De_{j_1},\De_{j_2})\ge 100C\de$, so we prove the disjointness.

As for those quadruples $\vec k$ in $A$ or $B$, we have
\begin{align}
\label{two-the-same}
    &\sum_{j_1\not=j_2}\big(\sum_{\vec k\in A}+\sum_{\vec k\in B}\big)\int (T_{j_1} P_{k_1}f\cdot T_{j_2} P_{k_2}f)(\overline{T_{j_1} P_{k_3}f\cdot T_{j_2} P_{k_4}f})\\
    \leq&\int \Big(\sum_{j}\sum_{k
    \in \ck}|T_j P_k f|^2\Big)\cdot\Big(\sum_{j}\Big|\sum_{k
    \in \ck}T_j P_k f\Big|^2\Big), 
\end{align}
which, by using H\"older's inequality, is bounded by 
\begin{equation}
    [(1/100)I+100II]/2.
\end{equation}
Note that for those quadruples $\vec k \in C$, 
$$ \sum_{j_1\not=j_2}\sum_{\vec k\in C}\int (T_{j_1} P_{k_1}f\cdot T_{T_2} P_{k_2}f)(\overline{T_{j_1} P_{k_3}f\cdot T_{j_2} P_{k_4}f}) \le II.$$
Thus we obtain
\begin{equation}
\label{off-diagonal-2}
    \eqref{off-diagonal-1}\leq (1/100)I+100II.
\end{equation}

\medskip
\textit{Case 2}: $j_1=j_2$.

When $j_1=j_2$, we similarly consider the subsets of $\ck^4$:
\begin{enumerate}
    \item $E_1=\{k_1=k_2\}$, $E_2=\{k_1=k_3\}$, $E_3=\{k_1=k_4\}$, $E_4=\{k_2=k_3\}$, $E_5=\{k_2=k_4\}$, $E_6=\{k_3=k_4\}$.
    \item $F_1'=\{k_1=k_2,k_3=k_4\}$,
    $F_2'=\{k_1=k_3,k_2=k_4\}$,
    $F_3'=\{k_1=k_4,k_2=k_3\}$.
    
    \item $F_1=\{k_1=k_2=k_3\}$, $F_2=\{k_1=k_2=k_4\}$, $F_3=\{k_1=k_3=k_4\}$, $F_4=\{k_2=k_3=k_4\}$.
    \item $G=\{k_1=k_2=k_3=k_4\}$.
    \item $H=\{k_1,k_2,k_3,k_4~\textup{are all different}\}$.
\end{enumerate}
Then by inclusion-exclusion principle, one can express $\ck^4$ as a linear combination of the aforementioned collection of quadruples: 
\begin{equation}
    \Id_{\ck^4}=\sum_{i=1}^6 \Id_{E_i}-\sum_{i=1}^3\Id_{F'_i}-2\sum_{i=1}^4\Id_{F_i}+6\cdot\Id_G+\Id_H.
\end{equation}
Note that when $j_1=j_2$ and $\vec k\in H$, \eqref{diagonal-1} is zero.  

As for those quadruples in $E_i$, similar to \eqref{two-the-same} we get
\begin{align}
    &\sum_{j_1=j_2}\sum_{\vec k\in E_i}\int (T_{j_1} P_{k_1}f\cdot T_{j_2} P_{k_2}f)(\overline{T_{j_1} P_{k_3}f\cdot T_{j_2} P_{k_4}f})\\
    \leq&\int \Big(\sum_{j}\sum_{k
    \in \ck}|T_j P_k f|^2\Big)\cdot\Big(\sum_{j}\Big|\sum_{k
    \in \ck}T_j P_k f\Big|^2\Big)\leq [(1/100)I+100II]/2.
\end{align}

For those $\vec k\in F_2'$ or $F_3'$, we have:
$$ \sum_{j_1=j_2}\sum_{\vec k\in F_2'(\textup{or} F_3')}\int (T_{j_1} P_{k_1}f\cdot T_{T_2} P_{k_2}f)(\overline{T_{j_1} P_{k_3}f\cdot T_{j_2} P_{k_4}f})=\int\sum_j(\sum_\ck|T_jP_kf|^2)^2 \le II.$$

For $\vec k\in F_1'$, we have: 
$$ \sum_{j_1=j_2}\sum_{\vec k\in F_1'}\int (T_{j_1} P_{k_1}f\cdot T_{T_2} P_{k_2}f)(\overline{T_{j_1} P_{k_3}f\cdot T_{j_2} P_{k_4}f})=\int\sum_j(\sum_\ck(T_jP_kf)^2)^2 \le II.$$

For those quadruples in $F_i$, for example in $F_1$, we get
\begin{align}
    &\sum_{j_1=j_2}\sum_{\vec k\in F_1}\int (T_{j_1} P_{k_1}f\cdot T_{j_2} P_{k_2}f)\overline{(T_{j_1} P_{k_3}f\cdot T_{j_2} P_{k_4}f})\\
    =&\sum_j\int \sum_{k_1}|T_{j} P_{k_1}f|^2\cdot T_{j} P_{k_1}f\cdot\sum_{k_4}\overline{T_{j} P_{k_4}f},
\end{align}
whose absolute value, by Cauchy-Schwartz inequality, is bounded above by 
\begin{align}
    &\sum_{j}\int\Big(\sum_{k_1}|T_{j} P_{k_1}f|^4\Big)^{1/2}\Big(\sum_{k_1} |T_{j} P_{k_1}f|^2\Big|\sum_{k_4}\overline{T_{j} P_{k_4}f}\Big|^2\Big)^{1/2}\\
    \leq  &\int\sum_j\sum_{k_1}|T_{j} P_{k_1}f|^4+\int\sum_j\sum_{k_1} |T_{j} P_{k_1}f|^2\Big|\sum_{k_4}T_{j} P_{k_4}f\Big|^2\\[1ex]
    \leq &\,II+[(1/100)I+100II]/2.
\end{align}

Finally, for those quadruples in $G$, we can easily get
\begin{align}
    &\sum_{j_1=j_2}\sum_{\vec k\in G}\int (T_{j_1} P_{k_1}f\cdot T_{j_2} P_{k_2}f)(\overline{T_{j_1} P_{k_3}f\cdot T_{j_2} P_{k_4}f})\\
    =&\int\sum_j\sum_{k}|T_{j} P_{k}f|^4\leq II.
\end{align}
Putting all estimates above together, we reach
\begin{align}
\label{diagonal-2}
    \eqref{diagonal-1}&\leq 3[(1/100)I+100II]+8II+4[(1/100)I+100II]+II\\
    &\leq (7/100)I+800II.
\end{align}

Now plug \eqref{off-diagonal-2} and \eqref{diagonal-2} \eqref{before-LP} so that
\begin{align}
    \eqref{before-LP}=I\leq (8/100)I+900II,
\end{align}
which gives $\eqref{before-LP}=I\leq 1000II$.
\end{proof}

Hence it remains to show
\begin{equation}
\label{reduction-1}
    \sqrt[4]{II}=\Big\|\Big(\sum_{j}\sum_{k
    \in \ck}|T_j P_kf|^2\Big)^{1/2}\Big\|_4\lesssim(\log \de^{-1})^{1/2}\|f\|_4.
\end{equation}
The desired estimate \eqref{reduction-1} can be further reduced to the following local estimate.
\begin{equation}
\label{reduction-2}
    \Big\|\Big(\sum_{j}|T_j P_1 f|^2\Big)^{1/2}\Big\|_4\lesssim(\log \de^{-1})^{1/4}\|f\|_4.
\end{equation}
It is a local version of \eqref{main-esti}. We also remark that in the local version, the constant $(\log\de^{-1})^{1/4}$ is better than the global version.
To deduce \eqref{reduction-1} from \eqref{reduction-2}, we need the following lemma. It is from \cite{guo2020sharp} Proposition 4.2 (see also \cite{jones2008strong} and \cite{seeger1988some}).

\begin{lemma}\label{reductionlem}
Let $\{m_j(\xi)\}_{j\in \mathcal J}$ be a set of Fourier multipliers on $\R^n$, each of which is compactly supported on $\{\xi:1/2\le |\xi|\le 2\}$, and satisfies
\begin{equation}\label{condition2}
\sup_{j\in\cj} |\partial^\al_\xi m_j(\xi)|\le B\ \textup{for~each~}0\le|\al|\le n+1
\end{equation}  
for some constant $B$. For $j\in\mathcal J$ and $k\in\ZZ$, write $T_{j,k}$ the multiplier operator with multiplier $m_j(2^{-k}\xi)$. Fix some $p\in [2,\infty)$. Assume that there exists some constant $A$ such that 
\begin{equation}\label{condition1}
    \sup_{k\in\ZZ}\big\| (\sum_{j\in\cj} |T_{j,k}f|^2)^{1/2} \big\|_{L^s(\R^n)}\le A\|f\|_{L^s(\R^n)}
\end{equation}
for both $s=p$ and $s=2$. Then
\begin{equation}\label{globales}
    \bigg\| (\sum_{k\in\ZZ}\sum_{j\in\cj}|T_{j,k}f|^2)^{1/2} \bigg\|_{L^p(\R^n)}\lesssim A\bigg|\log\bigg(2+\frac{B}{A}\bigg)\bigg|^{\frac{1}{2}-\frac{1}{p}}\|f\|_{L^p(\R^n)}.
\end{equation}
\end{lemma}

Let us discuss how to apply Lemma \ref{reductionlem}. First note that $$(T_jP_kf)^{\wedge}(\xi)=\Id^*_{\De_j}(\xi/|\xi|)\rho(2^{-k}|\xi|)\wh f(\xi).$$
If we assume \eqref{reduction-2} is true, by rescaling we have for any $k\in\ZZ$
\begin{equation}
    \Big\|\Big(\sum_{j}|T_j P_k f|^2\Big)^{1/2}\Big\|_4\lesssim(\log \de^{-1})^{1/4}\|f\|_4.
\end{equation}
We choose $m_j(\xi)=\Id^*_{\De_j}(\xi/|\xi|)\rho(\xi)$ in Lemma \ref{reductionlem}, so we have $T_{j,k}f=T_jP_kf$. We also choose $p=4$.
We can verify the constant $B=\de^{-O(1)}$ and $A=(\log\de^{-1})^{1/4}$ will make the conditions \eqref{condition2} and \eqref{condition1} in the lemma hold. As a result, from \eqref{globales} we obtain
$$ \Big\|\Big(\sum_{j}\sum_{k
    \in \ck}|T_j P_kf|^2\Big)^{1/2}\Big\|_4\lesssim A\bigg|\log\bigg(2+\frac{B}{A}\bigg)\bigg|^{1/4}\|f\|_4\lesssim (\log\de^{-1})^{1/2}\|f\|_4. $$
This gives the estimate \eqref{reduction-1}.

The proof of \eqref{reduction-2} is given in the next section.

\section{Proof of the local version}\label{sec3}
Recall we are given a set of disjoint $\de$-balls $\{\De_j\}$ in $\ZS^{n-1}$. Also recall the smooth Fourier restriction operator is defined as
$$ T_j f(x)=(\Id^*_{\De_j}(\xi/|\xi|)\wh f(\xi) )^\vee(x), $$
where $\Id^*_{\De_j}$ is a smooth bump function adapted to $\De_j$.

After the global-to-local reduction in the previous section, \eqref{main-esti} boils down to \eqref{reduction-2}, which is the following result
\begin{theorem}\label{localthm}
For any function $f$ with $\supp \wh f\subset \{\xi:100\le|\xi|\le 101\}$, we have
\begin{equation}\label{localineq}
    \bigg\|\big(\sum_{j}|T_j f|^2\big)^{1/2}\bigg\|_{L^4(\R^n)}\lesssim (\log \de^{-1})^{1/4} \|f\|_{L^4(\R^n)}, 
\end{equation} 
where $C>0$ is some universal constant.
\end{theorem}

Let us first discuss some geometry.
For each $\De_j$, we consider the corresponding tube $\tau_j$ defined as follows
$$ \tau_j:=\{\xi\in\R^n: 100\le|\xi|\le 101,~ \xi/|\xi|\in\De_j\}. $$
Since $\supp\wh f\subset \{\xi:100\le|\xi|\le 101\}$, we have 
$$T_j f=\bigg( \Id^*_{\De_j}(\xi/|\xi|) \rho(|\xi|) \wh f(\xi) \bigg)^\vee,$$
where $\rho(r)$ is a smooth function supported in $r\in [99,102]$ and $=1$ for $r\in[100,101]$.
Now we define 
$$\psi_{\tau_j}(\xi):=\Id^*_{\De_j}(\xi/|\xi|) \rho(|\xi|), $$
so $\psi_{\tau_j}$ is a smooth bump function supported in $\tau_j$, and $T_j f=\big( \psi_{\tau_j} \wh f\ \big)^\vee$.
Denote all the $\de\times\cdots\times\de\times1$ tubes obtained above by $\Tau=\{\tau_j\}$. By definition we know that the tubes are finitely overlapping.
Now, let us forget about $T_j f$ and use the new notation
$f_\tau:=\big( \psi_\tau \wh f\ \big)^\vee$ for $\tau\in\Tau$.
\eqref{localineq} is equivalent to
\begin{equation}
    \bigg\|\big(\sum_{\tau\in\Tau}|f_\tau|^2\big)^{1/2}\bigg\|_{L^4(\R^n)}\lesssim (\log \de^{-1})^{1/4} \|f\|_{L^4(\R^n)}
\end{equation}

For each $\tau\in\Tau$, the Fourier transform of $|f_\tau|^2$ has support in $\tau-\tau\subset 5\tau_0$ (here $\tau_0$ is the translation of $\tau$ to the origin). Hence the Fourier transform of $\sum_{\tau\in\Tau}|f_\tau|^2$ is supported in $\cup_{\tau\in\Tau} 5\tau_0\subset B_{10}(0)$. Next, we will partition the frequency ball $B_{10}(0)$ into tubes and analyze the contribution of $\sum_{\tau\in\Tau}|f_\tau|^2$ on each of the partitions.

For any dyadic number $s$ with $\de\le s\le 10$, consider a partition of the annulus $\bA_s:=\{\xi\in\R^n: \frac{s}{2}\le |\xi|\le s\}$ into tubes of dimensions $\de\times\cdots\times \de\times s$ whose central lines pass through the origin.
More precisely, we choose a set of maximal $\de s^{-1}$-separated points on $\ZS^{n-1}$, denoted by $\{\om_s\}$.
For each $\om_s$, define 
$$\theta_s=\{\xi: \frac{s}{2}\le|\xi|\le s,~ \xi/|\xi|\in \ZS^{n-1}\cap B_{\de s^{-1}}(\om_s)\}.$$
Denote the set of these tubes by $\Theta_s=\{\theta_s\}$. One can see that $\Theta_s$ forms a finitely overlapping covering of $\bA_s$. 
Particularly, when $s=\de$, we just define $\Theta_{\de}$ to consist of a single element $\theta_{\de}=\{\xi\in\R^n: |\xi|\le \de\}$
which is a ball of radius $\de$ centered at the origin.

Next, we will use $\theta_s\in\Theta_s$ to give a partition of $\Tau$. For each $\theta_s\in \Theta_s$, define
$$ \Tau_{\theta_s}:=\{\tau\in\Tau: \theta_s\cap 10\tau_0\neq \emptyset\}. $$
By some elementary geometries, we can see that $\{\Tau_{\theta_s}\}_{\theta_s\in\Theta_s}$ form a finitely overlapping cover of $\Tau$.

For each $\theta_s$, we define $\theta_s^*$ to be the dual slab of $\theta_s$. More precisely, $\theta_s^*$ is a slab centered at the origin with dimensions $\de^{-1}\times\cdots\times \de^{-1}\times s^{-1}$ and with the normal direction the same as the direction of $\theta_s$.

Now let us start the proof.
\begin{proof}
Denote the square function by $g=\sum_{\tau\in\Tau}|f_\tau|^2$. If $\xi\in\bA_s$, then
\begin{equation}
    |\wh g(\xi)|^2\lesssim \sum_{\theta_s\in\Theta_s} |\sum_{\tau\in\Tau_{\theta_s}}(|f_\tau|^2)^\wedge(\xi)|^2.
\end{equation}
For each $\theta_s$, we choose $\eta_{\theta_s}$ to be a smooth cutoff function at $\theta_s$ so that $\eta_{\theta_s}\gtrsim 1$ in $\theta_s$ and $|\eta_{\theta_s}^{\vee}(x)|\lesssim \frac{1}{|\theta_s^*|}1_{\theta_s^*}(x)$. Then, we have
\begin{equation}
    |\wh g(\xi)|^2\lesssim \sum_{\theta_s\in\Theta_s} |\eta_{\theta_s}(\xi)\sum_{\tau\in\Tau_{\theta_s}}(|f_\tau|^2)^\wedge(\xi)|^2,
\end{equation}
for $\xi\in\bA_s$.

By Plancherel, we get
\begin{equation}
    \int| g|^2\lesssim \int\sum_{\de\le s\le 10}\sum_{\theta_s\in\Theta_s} |\eta_{\theta_s}^{\vee}*\sum_{\tau\in\Tau_{\theta_s}}|f_\tau|^2|^2.
\end{equation}
Note that $|\eta_{\theta_s}^{\vee}(x)|\lesssim \frac{1}{|\theta_s^*|}1_{\theta_s^*}(x)$. This suggests us to tile $\R^n$ by translations of $\theta_s^*$. We denote this cover by $\{U: U\parallel\theta_s^*\}$. For $x\in U$, we have
$$ |\eta_{\theta_s}^{\vee}*\sum_{\tau\in\Tau_{\theta_s}}|f_\tau|^2|\lesssim |U|^{-1}\int \eta_U\sum_{\tau\in\Tau_{\theta_s}}|f_\tau|^2, $$
where $\eta_U(x)=\max_{y\in x+\theta_s^*}|\eta_{\theta_s}^\vee(x-y)|$ is a smooth bump function supported in $2U$.

Therefore, one has
\begin{equation}
    \int| g|^2\lesssim \sum_{\de\le s\le 10}\sum_{\theta_s\in\Theta_s}\sum_{U\parallel\theta_s^*}|U|^{-1}\big(\int_{2U}\sum_{\tau\in\Tau_{\theta_s}}|f_\tau|^2\big)^2.
\end{equation}
By dyadic pigeonholing, there exists an $s$ such that
\begin{equation}
    \int| g|^2\lesssim \log \de^{-1}\sum_{\theta_s\in\Theta_s}\sum_{U\parallel\theta_s^*}|U|^{-1}\big(\int_{2U}\sum_{\tau\in\Tau_{\theta_s}}|f_\tau|^2\big)^2.
\end{equation}
We remark that this is the only place we lose a logarithmic factor $\log\de^{-1}$.

\medskip

Now we carefully analyze the integral $\int_{2U}\sum_{\tau\in\Tau_{\theta_s}}|f_\tau|^2$.
For simplicity, we just write $\theta$ for $\theta_s$ (recall $\theta$ is a tube of dimensions $\de\times\cdots \de\times s$).
For each $\theta\in\Theta_s$ of form $\theta=\{\xi: \frac{s}{2}\le|\xi|\le s, \xi/|\xi|\in \ZS^{n-1}\cap B_{\de s^{-1}}(\om)\},$
we define another tube
\begin{equation}
    \vartheta:=\{ \xi: 100\le|\xi|\le 101, \xi/|\xi|\in\ZS^{n-1}\cap B_{100\de s^{-1}}(\om) \}.
\end{equation}
This roughly says the radial projection of $\theta$ on $\ZS^{n-1}$ is contained in that of $\vartheta$.
From a simple geometric argument, we see that each $\tau\in\Tau_\theta$ satisfies $\tau\subset \vartheta$.
Let $\bth=\{\vartheta\}$ be the collection of these $\vartheta$. For each $\vth$, we choose a smooth bump function $\psi_\vth$ at $\vth$ and then define $f_\vth:=(\psi_\vth\wh f\ )^\vee$.

By local $L^2$-orthogonality (see Lemma \ref{weightedL2}), we have
\begin{equation}
    \int_{2U}\sum_{\tau\in\Tau_{\theta_s}}|f_\tau|^2\lesssim \int \chi_U |f_\vth|^2,
\end{equation}
where $\chi_U$ is morally a cutoff function at $U$ and decay rapidly outside $U$.
Therefore, by H\"older's inequality, we have
\begin{equation}
    \int| g|^2\lesssim \log \de^{-1}\sum_{\theta_s\in\Theta_s}\sum_{U\parallel\theta_s^*}\int (\chi_U)^2 |f_\vth|^4\lesssim \log \de^{-1}\sum_{\theta_s\in\Theta_s}\int |f_\vth|^4.
\end{equation}
Here we use
$$ \sum_{U\parallel\theta_s^*}(\chi_U)^2\lesssim 1. $$
It remains to prove 
\begin{equation}
    \sum_{\theta_s\in\Theta_s}\int |f_\vth|^4\lesssim \int |f|^4.
\end{equation}
This is just a result of interpolation between the following two inequalities (or see Lemma \ref{interpolation}):
\begin{equation}
    \sup_{\theta_s\in\Theta_s}\| f_\vth\|_{L^\infty}\lesssim \|f\|_{L^\infty}.
\end{equation}
\begin{equation}
    \sum_{\theta_s\in\Theta_s} \|f_\vth\|_{L^2}^2\lesssim \|f\|_{L^2}^2.
\end{equation}
\end{proof}

\section{Proof of the sharp-cutoff version}
\label{sec4}
Let us prove Theorem \ref{main-esti-2}.
Now each $\De_j$ is $\de$-regular, so by definition there is a $\de$-ball $B_j\subset \ZS^{n-1}$ such that $cB_j\subset \De_j\subset CB_j$. By the disjointness of $\{\De_j\}$, we see the $C\de$-balls $\{CB_j\}$ are finitely overlapping. 
Without loss of generality, we may assume they are disjoint.
(Actually, we can regroup these balls into $O(1)$ sets so that the $C\de$-balls in each set are disjoint. Then we prove the estimate for each set and sum them up together.) We can choose smooth bump function $\Id^*_{B_j}$ adapted to $CB_j$ so that $\Id^*_{B_j}=1$ on $\De_j$. For any function $h$, we define $T_j h:= (\Id^*_{B_j} \wh h)^\vee$. By the support condition, we see that $S_j f= S_j T_j f$.

By duality, there is a $g\in L^2$ with $\|g\|_2=1$ so that
\begin{equation}\label{plugback2}
    \Big\|\big(\sum_{j}|S_j f|^2\big)^{1/2}\Big\|_4^2=\int_{\R^n}\sum_{j}|S_j f|^2g=\sum_j\int_{\R^n}|S_j T_j f|^2g.
\end{equation}
Let us look at the integral $\int_{\R^n} |S_j h|^2 g$, where we will plug in $h=T_j f$ later. Recall that $\wh{S_j h}= \Id_{\si_j} \wh h$, where $\si_j=\{\xi: \xi/|\xi|\in\De_j\}$. We want to express $\Id_{\si_j}$ in another form. Note that $\De_j$ is a polyhedron on $\ZS^{n-1}$, so if we denote by $\cn_j=\{\vec n_j\}$ the normals of $\De_j$ pointing outward, then 
$$ \Id_{\si_j}=\prod_{\vec n_j\in\cn_j} \Id_{\{\xi\cdot \vec n_j<0\}}. $$
By the one-dimensional condition (see Definition \ref{1dim}), we see that $\cn:=\cup_j \cn_j$ is contained in $O(1)$ many great circles. Define the following maximal functions
\begin{align}
    &M_{\vec n} g(x):=\sup_{t>0}\frac{1}{2t}\int _{x+[-t,t]\vec n}|g|,\\ 
    M_{s,\vec n} g(x):= &M_{\vec n}\big(( M_{\vec n} |g|^s )^{1/s}\big), \ \ \
    M_s g(x):=\sup_{\vec n\in\cn}M_{s,\vec n} g.
\end{align}
Here $s>1$. We see that $M_s$ is actually a maximal operator associated to one-dimensional directions.

We may assume $\cn$ is contained in one great circle and by rotation we may assume $\cn$ lie in the $(\xi_1,\xi_2)$-plane, i.e., $\cn\subset \{\xi\in\ZS^{n-1}:\xi_1^2+\xi_2^2=1\}$. Denote $x=(x_1,x_2,x')$, where $x'\in\R^{n-2}$. We write
\begin{align}
    \int_{\R^n}|S_j h|^2 g=\int_{\R^{n-2}}\int_{\R^2}|S_j h(x_1,x_2,x')|^2 g(x_1,x_2,x')dx_1 dx_2 dx'\\
    =\int_{\R^{n-2}}\int_{\R^2}|K(x_1,x_2)* h(x_1,x_2,x')|^2 g(x_1,x_2,x')dx_1 dx_2 dx'.
\end{align}
Here $K(x_1,x_2)=(\prod_{\vec n_j\in\cn_j} \Id_{\{ \xi\cdot\vec n_j<0 \}})^\vee$ is some kernel depending only on $x_1,x_2$.

By iterating the weighted estimate of C\'ordoba-Fefferman \cite{Cordoba-Fefferman},
we have for each $s>1$ the following estimate:
\begin{align}
    \int_{\R^2}|K(x_1,x_2)* h(x_1,x_2,x')|^2 g(x_1,x_2,x')dx_1 dx_2\\
    \lesssim \int_{\R^2}|h|^2  M_{s,\vec n_{j,1}}\circ\cdots \circ M_{s,\vec n_{j,m_j}} |g|.
\end{align}
Here $\cn_j=\{\vec n_{j,1},\cdots, \vec n_{j,m_j}\}$.
We assume $m_j\le m$ which is a bounded number, so the above inequality is bounded by
\begin{equation}
    \int_{\R^2}|h|^2  M_{s}^{(m)}|g|,
\end{equation}
where $M_{s}^{(m)}$ is the composition of $M_s$ by $m$ times. We obtain
$$ \int_{\R^n}|S_j h|^2 g\lesssim \int_{\R^n}|h|^2M_s^{(m)}|g|. $$

Plugging back to \eqref{plugback2}, we have
\begin{align}
    \Big\|\big(\sum_{j}|S_j f|^2\big)^{1/2}\Big\|_4^2
\lesssim \int_{\R^n}\sum_j |T_j f|^2 M_s^{(m)}|g|\\    
    \leq\Big\|\big(\sum_j|T_j f|^2\big)^{1/2}\Big\|_4^2\|(M_s^{(m)}|g|)^{1/s}\|_2.
\end{align}
Since $M_{s}^{(m)}$ is essentially a maximal operator in the plane, we can use the two-dimensional maximal estimate (see for example in \cite{Katz-two-d-directional}) and choose $s$ very close to $1$ to get $\big\|M_s^{(m)}|g|\big\|_2\lesssim_\e \de^{-\e}\|g\|_2$. This gives
\begin{equation}
    \Big\|\big(\sum_{j}|S_j f|^2\big)^{1/2}\Big\|_4^2\lesssim_\e \de^{-\e} \Big\|\big(\sum_{j}|T_j f|^2\big)^{1/2}\Big\|_4^2.
\end{equation}
This boils down to the smooth version that we have already proved in Section \ref{sec3}.

\section{\texorpdfstring{$L^8$}{Lg} square function estimate for the cone}\label{sec5}

We prove Theorem \ref{L8thm} and Theorem \ref{globalL8} in this section. Via a global-to-local reduction that is similar to the one in Section 2, we see that Theorem \ref{globalL8} is a corollary of Theorem \ref{L8thm}. Hence we will focus on the proof of Theorem \ref{L8thm}. Let us begin with some elementary tools.

\subsection{Some elementary estimates}

Let $R\subset\R^3$ be a rectangle of dimensions $a_1\times a_2\times a_3$. We will use $R^*$ to denote the dual rectangle of $R$, namely $R^*$ is the rectangle centered at the origin of dimensions $a_1^{-1}\times a_2^{-1}\times a_3^{-1}$. Also we make the convention that if $R$ lies in the physical space $\R^3_x$ then $R^*$ lies in the frequency space $\R^3_\xi$ and vice versa.

Sometimes we will use the notation that a function $\vp$ is a \textit{smooth bump function adapted} to $R$. What follows is its precise definition.

\begin{definition}
Let $R\subset \R^3_\xi$ be a rectangle of dimensions $a_1\times a_2\times a_3$ and let $(\boldsymbol e_1,\boldsymbol e_2,\boldsymbol e_3)$ be the corresponding directions. Denote by $\xi_R$ the center of $R$ and write $\xi$ in the coordinate $(\boldsymbol e_1,\boldsymbol e_2,\boldsymbol e_3)$ as $(\xi_1,\xi_2,\xi_3)$.
We say $\vp$ is ``a smooth bump function adapted" to $R$, if $\supp \vp\subset 2R$, $\vp=1$ on $R$ and $\vp$ satisfies the following derivative estimate
\begin{equation}\label{adapt}
    |D^k_{\boldsymbol e_i}\vp(\xi_R+\xi)|\lesssim_k (1+\frac{|\xi_i|}{a_i})^{-k},
\end{equation}
for $i=1,2,3$ and any $k>0$. 
\end{definition}

Following the notation in the definition above, if $\vp$ is a smooth bump function adapted to $R$, then
\begin{equation}
    |\vp^\vee(x)|\lesssim_k \frac{1}{|R^*|} (1+a|x_1|+b|x_2|+c|x_3|)^{-k}. 
\end{equation}
This roughly says $|\vp^\vee|\approx \frac{1}{|R^*|}\Id_{R^*}$, which suggests us to define the following indicator function with rapidly decaying tail.
\begin{definition}
Let $R$ be a rectangle and $A:\R^3\rightarrow\R^3$ be the linear map such that $A([-1,1]^3)=R$. We define
\begin{equation}
    \chi_R(x):=(1+|A^{-1}x|)^{-10^9}.
\end{equation} 
We call $\chi_R$ the indicator function of $R$ with rapidly decaying tail.
\end{definition}

The definitions above also work in $\R^n$. 
Now let us state a weighted $L^2$-estimate. For its proof, see \cite{Cordoba} or Lemma 2.3 in \cite{gan2021new}.

\begin{lemma}\label{weightedL2}
Let $\{R\}$ be a set of finitely overlapping congruent rectangles in $\R^n_\xi$, and let $\{\vp_R(\xi)\}_R$ be the smooth bump functions adapted to them. Then
\begin{equation}
    \int_{\R^n}\sum_R |(\vp_R \wh f)^\vee|^2 g\lesssim \int_{\R^n}|f|^2 (\frac{1}{|R^*|}\chi_{R^*}*|g|). 
\end{equation}
\end{lemma}

There is another useful lemma.
\begin{lemma}\label{interpolation}
Let $\{R\}$ be a set of finitely overlapping rectangles in $\R^n_\xi$, and let $\{\vp_R(\xi)\}_R$ be the smooth bump functions adapted to them. Then
\begin{equation}
    \int_{\R^n}\sum_R |(\vp_R \wh f)^\vee|^p \lesssim_p \int_{\R^n}|f|^p, 
\end{equation}
for $2\le p\le \infty$.
\end{lemma}
Lemma \ref{interpolation} follows from interpolation between $p=2$ and $p=\infty$. We also have the local version of Lemma \ref{interpolation}.

\begin{lemma}\label{localinterp}
Let $\{R\}$ be a set of finitely overlapping congruent rectangles in $\R^n_\xi$, and let $\{\vp_R(\xi)\}_R$ be the smooth bump functions adapted to them. If $U\subset \R^n_x$ is a rectangle whose translation to the origin contains all the dual rectangles $R^*$, then we have the following estimate:
\begin{equation}
    \int \chi_{U}\sum_R |(\vp_R \wh f)^\vee|^p \lesssim_p \int \chi_{U}|f|^p, 
\end{equation}
for $2\le p\le \infty$.
\end{lemma}
Again, the proof is by interpolation between $p=2$ and $p=\infty$. The case $p=\infty$ is easy, let us focus on $p=2$. We may assume $U$ is centered at the origin. Since $R^*\subset U$, we can partition each $R$ into smaller rectangles that are comparable to $U^*$. Denote the set of all the smaller rectangles coming from the partition by $\{\om\}$. By local orthogonality, we have
$$ \int \chi_{U}\sum_R |(\vp_R \wh f)^\vee|^2\lesssim \int \chi_{U}\sum_\om |(\vp_\om \wh f)^\vee|^2, $$
where $\vp_\om$ are smooth bump functions adapted to $\om$.
Together with Lemma \ref{weightedL2} and the fact that $\frac{1}{|U|}\chi_{U}*\chi_{U}\lesssim \chi_{U}$, we prove the result.

\begin{remark}

\rm 
From the point of view of the so-called ``locally constant property",
the lemmas above are obvious, but we still state them carefully for rigorousness.
\end{remark}

\subsection{The cutoff replacing property} \label{replace}
Let $R, R'\subset \R^3_\xi$ be two rectangles and $\vp_R,\vp_{R'}$ be smooth bump functions adapted to them. If $\supp \wh f\subset R\cap R'$, then $(\vp_R \wh f)^{\vee}=(\vp_{R'} \wh f)^{\vee}$. In this case, we replace the cutoff $\vp_R$ by $\vp_{R'}$, so we call it  \textit{the cutoff replacing property}.

When $\supp \wh f$ is not contained in $R\cap R'$, we can still have the cutoff replacing property by choosing $\vp_R, \vp_{R'}$ carefully. In the following, we discuss the case that we need in the paper. Let $R=[-a,a]\times [-b,b]\times [-c,c], R'=[-a,a]\times [-b,b]\times [-c',c']$ be two rectangles with $a>b>c'>c$ in the frequency space $\R^3_\xi$. Let $f$ be a function in $\R^3_x$ so that $\supp \wh f\subset \R^3_\xi\setminus \{|\xi_1|\le a, |\xi_2|\le b, |\xi_3|>c \}$. We see that $\supp \wh f$ is not contained in $R\cap R'$, but we can still construct the smooth bump functions to satisfy the property.

Choose $\vp(\xi_1,\xi_2)$ to be a smooth bump function adapted to $[-a,a]\times [-b,b]$. Choose $\rho(\xi_3)$ to be a smooth bump function adapted to $[-c,c]$, then $\rho(\frac{c}{c'}\xi_3)$ is a smooth bump function adapted to $[-c',c']$. If we set $\Id^*_R= \vp(\xi_1,\xi_2)\rho(\xi_3)$ and $\Id^*_{R'}=\vp(\xi_1,\xi_2)\rho(\frac{c}{c'}\xi_3)$, we see that $\Id^*_R, \Id^*_{R'}$ are adapted to $R, R'$, and $\Id^*_{R}=\Id^*_{R'}$ on $\R^3_\xi\setminus \{|\xi_1|\le a, |\xi_2|\le b, |\xi_3|>c \}$ which contains $\supp \wh f$. So, we have
$(\Id^*_R\wh f)^\vee=(\Id^*_{R'}\wh f)^\vee$.

Let us discuss how it works in our proof. Let $\tau$ be a $\de\times\de^{1/2}\times 1$-plank in $N_\de(\Ga)$ and let $\om$ be a $\ga^{-1} \de\times\de^{1/2}\times 1$-plank which is the $\ga^{-1}$-dilation of $\tau$ in the shortest direction (here $\ga<1$). Our function $f$ satisfies $\supp \wh f\subset N_\de(\Ga)$, so it also satisfies a similar condition discussed above after rotation. So, we can find $\Id^*_\tau$ adapted to $\tau$ and $\Id^*_\om$ adapted to $\om$ such that $(\Id^*_\tau\wh f)^\vee=(\Id^*_\om \wh f)^\vee$.  

\begin{remark}

\rm

The reason we want to change the cutoff is that we want to apply Lemma \ref{weightedL2} and Lemma \ref{localinterp}.
\end{remark}

\subsection{A general version of square function}
We start the proof of Theorem \ref{L8thm}.
Recall that $\Ga=\{ \xi\in\R^3:\xi_1^2+\xi_2^2=\xi_3^2, 1/2\le |\xi_3|\le 1 \}$, and $N_\de (\Ga)$ is its $\de$-neighborhood. There is a canonical covering of $N_\de(\Ga)$ using finitely overlapping planks $\tau$ of dimensions $\sim \de\times \de^{1/2}\times 1$, denoted by $\Tau=\{\tau\}$. For each $\tau\in\Tau$, choose a smooth bump function $\Id^*_\tau$ adapted to $\tau$. Define $f_\tau:=(\Id^*_\tau \wh f\ )^\vee$ as usual.

To prove the theorem, we need to state the estimate in a more technical way. Fix a dyadic parameter $\de^{1/2}\le \ga\le 1$. For each $\tau\in\Tau$, we partition it into planks of dimensions $\de\times \de^{1/2}\times \ga$ along the longest side of $\tau$ (see Figure \ref{finerpartition}). Denote the collection of these sub-planks of $\tau$ by $\Th_\ga(\tau)$. For each $\theta\in\Th_\ga(\tau)$, there is a smooth cutoff function $\Id^*_\theta$ adapted to $\theta$ and meanwhile we have $\Id^*_\tau=\sum_{\theta\in\Th_\ga(\tau)}\Id^*_\theta$.
We define $f_\theta=(\Id^*_\theta \wh f\ )^\vee$, then $f_\tau=\sum_{\theta\in\Th_\ga(\tau)}f_\theta$.
Set $\Theta_\ga=\cup_{\tau\in\Tau}\Th_\ga(\tau)$ be the set of all these planks.

\begin{figure}
\begin{tikzpicture}

\draw[color=blue, thick] (-9.6,9.16) arc (120:90:10);
\draw[color=blue, thick] (-10,8.66) arc (120:90:10);
\draw[color=blue, thick, dotted] (-9.5,7.794) arc (120:90:9);
\draw[color=blue, thick, dotted] (-9,6.96) arc (120:90:8);
\draw[color=blue, thick] (-8.5,6.06) arc (120:90:7);

\draw[color=blue,rotate around={120:(-5,0)}] (2,0) -- (5,0);
\draw[color=blue,rotate around={110:(-5,0)}] (2,0) -- (5,0);
\draw[color=blue,rotate around={100:(-5,0)}] (2,0) -- (5,0);
\draw[color=blue,rotate around={90:(-5,0)}] (2,0) -- (5,0);
\draw[color=blue,rotate around={87.8:(-5,0)}] (2.3,0) -- (5.5,0);

\draw[color=blue, thick] (-9.6,9.16) -> (-10,8.66);
\draw[color=blue, thick] (-8,9.9) -> (-8.42,9.397);
\draw[color=blue, thick] (-6.3,10.35) -> (-6.7365,9.848);
\draw[color=blue, thick] (-4.6,10.5) -> (-5,10);

\draw[color=blue, thick, dotted] (-4.63,9.45) -> (-5,9);
\draw[color=blue, thick, dotted] (-4.67,8.4) -> (-5,8);
\draw[color=blue, thick] (-4.72,7.35) -> (-5,7);

\node at (-6,9.5) {$\theta$};
\node at (-6,11) {$\tau$};
\node at (-2,9) {$\theta\sim\de\times\de^{1/2}\times\ga$};

\end{tikzpicture}
\caption{Finer partition}
\label{finerpartition}
\end{figure}
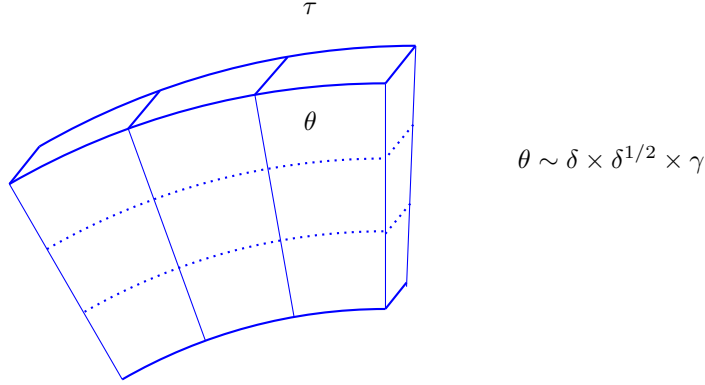

We will prove the following general version of the square function estimate.
\begin{theorem}
Assuming $\wh f\subset N_\de(\Ga)$, then for any $\e>0$ we have
\begin{equation}\label{genL8}
     \bigg\| (\sum_{\theta\in\Th_\ga}|f_\theta|^2)^{1/2} \bigg\|_{L^8(\R^3)}\le C_\e (\ga \de^{-1})^\e \|f\|_{L^8(\R^3)}.
\end{equation}
\end{theorem}

\begin{proof}
We prove by induction on $\ga$. The base case is when $\ga=\de^{1/2}$. For each $\theta\in\Theta_{\de^{1/2}}$, there is a cube $Q_\theta$ of side length $\de^{1/2}$ such that $\theta\subset Q_\theta$. By the cutoff replacing property in Section \ref{replace}, we can choose a smooth bump function $\Id^*_{Q_\theta}$ adapted to $Q_\theta$ so that $\Id^*_\theta \wh f=\Id^*_{Q_\theta} \wh f$. Via the Littlewood-Paley theory for congruent cubes (see \cite{Rubio} or \cite{Lacey-LP}), we obtain
\begin{equation}
     \bigg\| (\sum_{\theta\in\Th_{\de^{1/2}}}|f_\theta|^2)^{1/2} \bigg\|_{L^p(\R^3)}\lesssim  \|f\|_{L^p(\R^3)}
\end{equation}
for any $1<p<\infty$, which in particular implies \eqref{genL8} when $\ga=\de^{1/2}$.

Assuming \eqref{genL8} is proved for $\ga\le\ga_0/2$,  we are going to look at the case $\ga=\ga_0$. We suggest the reader to pretend $\ga_0=1$ for the first time of reading. For simplicity, we will omit the subscript $\ga$ and write $\Th_\ga$ (or $\Th_\ga(\tau)$) as $\Th$ (or $\Th(\tau)$).
For each $\tau\in\Tau$, define the square function
\begin{equation}\label{defgtau}
    g_\tau:=\sum_{\theta\in\Theta(\tau)}|f_\theta|^2.
\end{equation}
Each term in the summation has Fourier transform supported in $\theta-\theta$ which is roughly the translation of $\theta$ to the origin. Note that all the $\theta\in\Th(\tau)$ are roughly a translation of each other, so there is a plank of dimensions $\sim\de\times\de^{1/2}\times\ga$ centered at the origin such that $\theta-\theta$ is contained in this plank for any $\theta\in\Th(\tau)$. We denote this plank by $\theta_\tau$. Now we obtain functions $\{g_\tau\}_{\tau\in\Tau}$ and planks $\{\theta_\tau\}_{\tau\in\Tau}$ with $\supp\wh g_\tau\subset \theta_\tau$. We also see that 
$$ \sum_{\theta\in\Th}|f_\theta|^2=\sum_\tau g_\tau. $$
Our goal is to estimate $\int_{\R^3} |\sum_\tau g_\tau|^4$.

Next, we will do a high-low frequency decomposition for each $g_\tau$. Fix a large enough constant $K$ which is to be determined later. Note that $\theta_\tau$ is a plank of dimensions $\de\times\de^{1/2}\times\ga$ centered at the origin. Define another plank which we call the ``low plank" as: $\theta_{\tau,low}:=\theta_{\tau}\cap B_{CK^{-1}\ga}(0)$. Roughly speaking, $\theta_{\tau,low}$ is the portion of $\theta_\tau$ that is centered at the origin and has dimensions $\de\times\de^{1/2}\times K^{-1}\ga$. Choose a smooth bump function $\Id^*_{\theta_\tau}$ adapted to $\theta_\tau$ and a smooth bump function $\Id^*_{\theta_{\tau,low}}$ adapted $\theta_{\tau,low}$. Define
\begin{equation}\label{defg}
    g_{\tau,low}:=(\Id^*_{\theta_{\tau,low}}\wh g_\tau)^\vee,\ \ \   g_{\tau,high}:=\big((\Id^*_{\theta_{\tau}}-\Id^*_{\theta_{\tau,low}})\wh g_\tau\big)^\vee.
\end{equation}
Since $\supp \wh g_\tau\subset \theta_\tau$, we have
$$ g_{\tau,low}+g_{\tau,high}=(\Id^*_{\theta_\tau}\wh g_\tau)^\vee=g_\tau. $$
By triangle inequality, we have
\begin{equation}\label{hl}
    \int_{\R^3} |\sum_{\tau}g_\tau|^4\lesssim \int_{\R^3} |\sum_{\tau}g_{\tau,low}|^4+\int_{\R^3} |\sum_{\tau}g_{\tau,high}|^4.
\end{equation}
We call the two terms on the right hand side above low term and high term. We consider them separately.\\

\textit{Estimate for the low term}: 
For each $\theta\in\Th$, we cover it by finitely overlapping planks of dimensions $\de\times\de^{1/2}\times K^{-1}\ga$, denoted by $\{\theta'\}$. We use ``$\theta'<\theta$" to indicate that $\theta'$ comes from the covering of $\theta$. One observation is that: if $\theta\in\Theta(\tau)$ and $\theta'<\theta$, then $\theta'$ is roughly a translation of $\theta_{\tau,low}$. For fixed $\theta$ and $\{\theta'\}_{\theta'<\theta}$, we choose smooth functions $\{\Id^*_{\theta'}\}_{\theta'<\theta}$ so that each $\Id^*_{\theta'}$ is a smooth bump function adapted to $\theta'$ and
\begin{equation}
    \Id^*_{\theta}=\sum_{\theta'<\theta}\Id^*_{\theta'}
\end{equation}
on $N_\de(\Ga)$.

As usual, we define $f_{\theta'}:=(\Id^*_{\theta'}\wh f\ )^\vee$. Since $\supp\wh f \subset N_\de(\Ga)$, we have
\begin{equation}
    f_\theta=\sum_{\theta'<\theta}f_{\theta'}.
\end{equation}
Let us look at the low term. By definition,
\begin{align}
    \wh g_{\tau,low}=\Id^*_{\theta_{\tau,low}} \sum_{\theta\in\Th(\tau)}\wh{|f_\theta|^2}=\Id^*_{\theta_{\tau,low}} \sum_{\theta\in\Th(\tau)}\big(|\sum_{\theta'<\theta}f_{\theta'}|^2\big)^{\wedge}\\
    =\Id^*_{\theta_{\tau,low}} \sum_{\theta\in\Th(\tau)}\sum_{\theta_1',\theta_2'<\theta}\wh f_{\theta_1'}*\wh {\overline{f_{\theta_2'}}}. 
\end{align}

Note that $\supp (\Id^*_{\theta_\tau,low}\cdot \wh f_{\theta_1'}*\wh {\overline{f_{\theta_2'}}})$ is contained in $\theta_\tau\cap (\theta_1'-\theta_2')$, so $\Id^*_{\theta_\tau,low}\cdot \wh f_{\theta_1'}*\wh {\overline{f_{\theta_2'}}}$ is nonzero only when $\theta_1'$ and $\theta_2'$ are roughly adjacent. We use $\theta_2'\sim \theta_1'$ to denote that $\theta_2'$ and $\theta_1'$ are adjacent. As a result, we have
\begin{equation}
    \wh g_{\tau,low}=\Id^*_{\theta_{\tau,low}}\sum_{\theta\in\Th(\tau)}\big( \sum_{\theta'_1<\theta}\sum_{\theta_2' \sim\theta_1'} f_{\theta_1'}\overline{f_{\theta_2'}} \big)^\wedge.
\end{equation}
So, we have
\begin{align}
    g_{\tau,low}&=(\Id^*_{\theta_{\tau,low}})^{\vee}*\sum_{\theta\in\Th(\tau)} \sum_{\theta'_1<\theta}\sum_{\theta_2' \sim\theta_1'} f_{\theta_1'}\overline{f_{\theta_2'}}\\
    &\lesssim |(\Id^*_{\theta_{\tau,low}})^{\vee}|*\sum_{\theta\in\Th(\tau)} \sum_{\theta'<\theta}| f_{\theta'}|^2\\
    \label{localconst}&\lesssim \frac{1}{|\theta_{\tau,low}^*|}\chi_{\theta_{\tau,low}^*}*\sum_{\theta\in\Th(\tau)} \sum_{\theta'<\theta}| f_{\theta'}|^2.
\end{align}
where the second-last inequality is by the fact that for each $\theta_1'$ there are $O(1)$ many $\theta_2'$ adjacent to $\theta_1'$.

Since $\supp \wh f_{\theta'}\subset \theta'$, we
have $|f_{\theta'}|^2$ is locally constant on any translation of $\theta'^*$. Also noting that $\theta_{\tau,low}^*$ and $\theta'^*$ are roughly the same, we have $|f_{\theta'}|$ is locally constant on any translation of $\theta_{\tau,low}^*$. So, we actually have
\begin{equation}\label{heuristicc}
    \eqref{localconst}\approx \sum_{\theta\in\Th(\tau)} \sum_{\theta'<\theta}| f_{\theta'}|^2=\sum_{\theta'\in\Th_{\ga K^{-1}}}|f_{\theta'}|^2.
\end{equation}
By induction hypothesis, we have the following estimate for the low term.
\begin{equation}\label{lowest}
    \int_{\R^3}|\sum_\tau g_{\tau,low}|^4 =\int_{\R^3}(\sum_{\theta\in\Th_{\ga K^{-1}}}|f_{\theta'}|^2)^4\le \big(C_\e (\ga K^{-1}\de^{-1})^\e\|f\|_{L^8}\big)^8.
\end{equation}



\medskip

\textit{Estimate for the high term}: We consider the truncated cone 
\begin{equation}\label{defcone}
    \Ga_\ga:=\{ \xi\in\R^3: \xi_1^2+\xi_2^2=\xi_3^2, \ga/K\le |\xi_3|\le \ga \}, 
\end{equation} 
and its $C\ga^{-1}\de$-neighborhood $N_{C\ga^{-1}\de}\Ga_\ga$. For simplicity, we will omit the constant $C$ and just write as $N_{\ga^{-1}\de}\Ga_\ga$. By definition, the support of $\wh g_{\tau,high}$ is contained in $\theta_\tau\setminus\theta_{\tau,low}$ which consists of two planks symmetric with respect to the origin and of dimensions  $\de\times\de^{1/2}\times \ga$. We denote them by $\theta_\tau\setminus\theta_{\tau,low}=\theta_{\tau,high}^+\cup\theta_{\tau,high}^-$, where $\theta_{\tau,high}^+$ lies in $\{\xi_3>0\}$ and $\theta_{\tau,high}^-$ lies in $\{\xi_3<0\}$. By a simple geometry, we see that $\supp\wh g_{\tau,high}\subset N_{\ga^{-1}\de}\Ga_\ga$. We choose a finitely overlapping covering of $N_{\ga^{-1}\de}\Ga_\ga$ by $\ga^{-1}\de\times \de^{1/2}\times \ga$-planks $\omega$, denoted by:
$$ N_{\ga^{-1}\de}\Ga_\ga=\bigcup\omega. $$
For each $\wh g_{\tau,high}$, there exists $\omega$ such that $\theta_{\tau,high}^+\cup\theta_{\tau,high}^-\subset \omega\cup \omega_{re}$ and hence
$\supp\wh g_{\tau,high}\subset \omega\cup \omega_{re}$. Here $\om_{re}=\{-\xi: \xi\in\om\}$ denotes the reflection of $\om$ with respect to the origin. We associate $\tau$ to $\om$ (if there are multiple choices, we choose one). The reader can also check each $\supp \wh g_{\tau,high}$ can intersect a bounded number of sets from $\{\om\cup \om_{re}\}_\om$. The relation between $\om$ and $\theta_{\tau,high}^+$ is given in Figure \ref{theta-omega}.

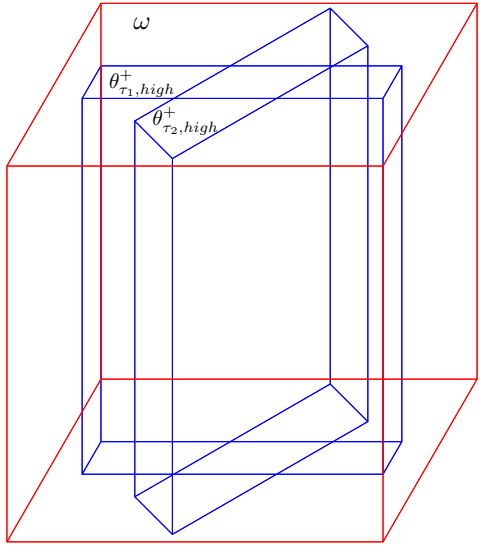
\begin{figure}
\begin{tikzpicture}
 \begin{scope}
   [x={(4cm,0cm)},
    y={({cos(60)*.5cm},{sin(60)*.5cm})},
    z={({cos(90)*5cm},{sin(90)*5cm})},line join=round,fill opacity=0.5,blue]
  \draw[] (0,0,0) -- (0,0,1) -- (0,1,1) -- (0,1,0) -- cycle;
  \draw[] (0,0,0) -- (1,0,0) -- (1,1,0) -- (0,1,0) -- cycle;
  \draw[] (0,1,0) -- (1,1,0) -- (1,1,1) -- (0,1,1) -- cycle;
  \draw[] (1,0,0) -- (1,0,1) -- (1,1,1) -- (1,1,0) -- cycle;
  \draw[] (0,0,1) -- (1,0,1) -- (1,1,1) -- (0,1,1) -- cycle;
  \draw[] (0,0,0) -- (1,0,0) -- (1,0,1) -- (0,0,1) -- cycle;
 \end{scope}
  \begin{scope}
   [shift={(1.2,-.8)},x={(-.5cm,.5cm)},
    y={({cos(30)*3cm},{sin(30)*3cm})},
    z={({cos(90)*5cm},{sin(90)*5cm})},line join=round,fill opacity=0.5,blue]
  \draw[] (0,0,0) -- (0,0,1) -- (0,1,1) -- (0,1,0) -- cycle;
  \draw[] (0,0,0) -- (1,0,0) -- (1,1,0) -- (0,1,0) -- cycle;
  \draw[] (0,1,0) -- (1,1,0) -- (1,1,1) -- (0,1,1) -- cycle;
  \draw[] (1,0,0) -- (1,0,1) -- (1,1,1) -- (1,1,0) -- cycle;
  \draw[] (0,0,1) -- (1,0,1) -- (1,1,1) -- (0,1,1) -- cycle;
  \draw[] (0,0,0) -- (1,0,0) -- (1,0,1) -- (0,0,1) -- cycle;
 \end{scope}
  \begin{scope}
   [shift={(-1,-.9)},x={(5cm,0cm)},
    y={({cos(60)*2.5cm},{sin(60)*2.5cm})},
    z={({cos(90)*5cm},{sin(90)*5cm})},line join=round,fill opacity=0.5,blue]
  \draw[red] (0,0,0) -- (0,0,1) -- (0,1,1) -- (0,1,0) -- cycle;
  \draw[red] (0,0,0) -- (1,0,0) -- (1,1,0) -- (0,1,0) -- cycle;
  \draw[red] (0,1,0) -- (1,1,0) -- (1,1,1) -- (0,1,1) -- cycle;
  \draw[red] (1,0,0) -- (1,0,1) -- (1,1,1) -- (1,1,0) -- cycle;
  \draw[red] (0,0,1) -- (1,0,1) -- (1,1,1) -- (0,1,1) -- cycle;
  \draw[red] (0,0,0) -- (1,0,0) -- (1,0,1) -- (0,0,1) -- cycle;
 \end{scope}

\node[scale=0.8] at (1.38,4.7) {$\theta_{\tau_2,high}^+$};
\node[scale=0.8] at (.8,5.2) {$\theta_{\tau_1,high}^+$};
\node at (.8,6) {$\om$}; 
 
\end{tikzpicture}
\caption{Relation between $\theta_{\tau,high}^+$ and $\om$}
\label{theta-omega}
\end{figure}

\begin{remark}

\rm

It's not harmful to the proof if we ignore $\theta_{\tau,high}^-$ (the other end of $\theta_\tau\setminus \theta_{\tau,low}$) and think of $\supp \wh g_{\tau,high}\subset\theta_{\tau,high}^+$. It's convenient to assume all the $\om$ lie in the upper half space $\{\xi_3>0\}$.
\end{remark}

Define $\Tau(\om)$ to be the set of $\tau$'s that are associated to $\om$.
For each $\om$, we define our function
\begin{equation}\label{defh}
    h_\om:=\sum_{\tau\in\Tau(\om)}g_{\tau,high}. 
\end{equation} 
We see that $\supp\wh h_{\om}\subset \om\cup\om_{re}$.
We have
\begin{equation}
    \int_{\R^3}|\sum_\tau g_{\tau,high}|^4=\int_{\R^3}|\sum_\om h_\om|^4.
\end{equation}
By rescaling of the factor $\ga^{-1}$ in all directions, we see that $N_{\ga^{-1}\de}\Ga_\ga$ becomes $N_{\ga^{-2}\de}\Ga_1$ and $\om$ becomes a $\ga^{-2}\de\times \ga^{-1}\de^{1/2}\times 1$-plank. We want to apply Guth-Wang-Zhang's reverse square function estimate \eqref{reverse}.
Before doing so, we give a remark.
\begin{remark}
\rm
In the setting of \eqref{reverse}, we use the cone $\Ga=\{\xi_1^2+\xi_2^2=\xi_3^3, 1/2\le|\xi_3|\le 1\}$, but we can also use the cone $\Ga_1$ defined in \eqref{defcone} at the cost of a factor $K^{O(1)}$ in \eqref{reverse}. That is 
$$ \|\sum_{\tau}f_\tau\|_{L^4(\R^3)}\le C_{\e'} K^{O(1)}\de^{-\e'} \bigg\| (\sum_\tau|f_\tau|^2)^{1/2} \bigg\|_{L^4(\R^3)}.
 $$
\end{remark}
We have from \eqref{reverse} that
\begin{equation}\label{comeback}
    \int_{\R^3}|\sum_\tau g_{\tau,high}|^4=\int_{\R^3}|\sum_\om h_\om|^4\le \big(C_{\e'} K^{O(1)}(\ga^{-2}\de )^{-\e'}\big)^4\int_{\R^3}(\sum_\om|h_\om|^2)^2.
\end{equation}
To estimate $\int_{\R^3}(\sum_\om|h_\om|^2)^2$, we will use a general version of Lemma 1.4 in \cite{guth2020sharp}. Before stating the lemma, we introduce some notations. 

Fix $R=\ga^2\de^{-1}$. We see that $\{\om\}$ are $R^{-1}\ga\times R^{-1/2}\ga\times \ga$-planks that form a finitely overlapping covering of $N_{R^{-1}\ga}(\Ga_\ga)$. Suppose we are given a set of functions $\vh=\{h_\om\}_\om$ with $\supp \wh h_\om\subset \om$.
For any dyadic $s$ in the range $R^{-1/2}\le s\le 1$, let $\Om_s=\{\om_s\}$ be $s^2\ga\times s\ga\times \ga$-planks that form a finitely overlapping covering of $N_{s^2\ga}(\Ga_\ga)$. For any $\om_s$, we define $U_{\om_s}$ to be the  plank centered at the origin  of dimensions $R\ga^{-1}\times Rs\ga^{-1}\times Rs^2\ga^{-1}$. Here, the edge of $U_{\om_s}$ with length $R\ga^{-1}$ (respectively $Rs\ga^{-1}, Rs^2\ga^{-1}$) has the same direction as the edge of $\om_s$ with length $s^2\ga$ (respectively $s\ga,\ga$). The motivation for the definition of $U_{\om_s}$ is that $U_{\om_s}$ is the smallest plank that contains the dual plank of $\om$ for all $\om\subset \om_s$. Later we will do rescaling $\xi\mapsto \ga^{-1}\xi$ in the frequency space, so that $\{\om\}$ becomes standard $R^{-1}\times R^{-1/2}\times 1$-planks that cover the $N_{R^{-1}}(\Ga_1)$.

We tile $\R^3$ by translated copies of $U_{\om_s}$. We write $U\parallel U_{\om_s}$ to denoted that $U$ is one of the copies, and define $S_U \vh$ by 
\begin{equation}
    S_U \vh=\big( \sum_{\om\subset \om_s}|h_\om|^2\big)^{1/2}|_U.
\end{equation}
\begin{remark}
\rm

Different from \cite{guth2020sharp}, we don't require there is a common function $h$ for all the $h_\om$ so that $h_\om$ is the Fourier restriction of $h$ to $\om$.
The only condition we need is $\supp \wh h_\om\subset \om$. One can compare the notation $S_U \vh$ here to a similar one in \cite{guth2020sharp}.
\end{remark}

We state the following lemma which is a general version of Lemma 1.4 in\cite{guth2020sharp}.
\begin{lemma}\label{kakeyalem}
Let the notation be given above. Then we have
\begin{equation}
    \int_{\R^3} (\sum_\om |h_\om|^2)^2\lesssim \sum_{R^{-1/2}\le s\le 1}\sum_{\om_s\in\Om_s}\sum_{U\parallel U_{\om_s}}|U|^{-1}\|S_U \vh\|_{L^2}^4.
\end{equation}
\end{lemma}
To prove Lemma \ref{kakeyalem}, we first do the rescaling $\xi\mapsto \ga^{-1}\xi$ in the frequency space and correspondingly $x\mapsto \ga x$ in physical space. Then, the proof is identical to the proof of Lemma 1.4 in \cite{guth2020sharp} which we do not reproduce here.

\medskip

Let us come back to \eqref{comeback}. By Lemma \ref{kakeyalem}, we have
\begin{align}
    \int_{\R^3}(\sum_\om|h_\om|^2)^2&\lesssim  \sum_{R^{-1/2}\le s\le 1}\sum_{\om_s\in\Om_s}\sum_{U\parallel U_{\om_s}}|U|^{-1}\|S_U \vh\|_{L^2}^4\\
    &=\sum_{R^{-1/2}\le s\le 1}\sum_{\om_s\in\Om_s}\sum_{U\parallel U_{\om_s}}|U|^{-1}\big(\int_U \sum_{\om\subset\om_s}|h_\om|^2 \big)^2.
\end{align}
Recalling the definition \eqref{defh} and \eqref{defg}, the above formula equals
\begin{equation}\label{aboveformula}
    \sum_{R^{-1/2}\le s\le 1}\sum_{\om_s\in\Om_s}\sum_{U\parallel U_{\om_s}}|U|^{-1}\big(\int_U \sum_{\om\subset\om_s}|\sum_{\tau\in\Tau(\om)} (\Id^*_{\theta_\tau}-\Id^*_{\theta_{\tau,low}})^\vee*g_\tau |^2 \big)^2.
\end{equation}

Before proceeding further, we give another definition. For an $\om$, we see that $\cup_{\tau\in\Tau(\om)} \tau$
is a union of $\sim\ga^{-1}$ many $\de\times\de^{1/2}\times 1$-planks, so $\cup_{\tau\in\Tau(\om)} \tau$ is contained in a $\de\ga^{-2}\times \de^{1/2}\ga^{-1}\times 1$-plank. By abuse of notation, we also use $\Tau(\om)$ to denote this plank. We define $f_{\Tau(\om)}$ to be the smooth Fourier restriction of $f$ to this plank, i.e., $f_{\Tau(\om)}:=(\Id^*_{\Tau(\om)}\wh f)^\vee$, where $\Id^*_{\Tau(\om)}$ is a smooth bump function adapted to $\Tau(\om)$. The property we will use is: $\Id^*_\theta=\Id^*_\theta\Id^*_{\Tau(\om)}$ if $\tau\in\Tau(\om)$ and $\theta\in\Theta(\tau)$. Heuristically, one may think $f_{\Tau(\om)}=\sum_{\tau\in\Tau(\om)}f_\tau$.

Fix an $U\parallel U_{\om_s}$. Let us look at the integral $\int_U |\sum_{\tau\in\Tau(\om)} (\Id^*_{\theta_\tau}-\Id^*_{\theta_{\tau,low}})^\vee*g_\tau |^2 $ in the above formula \eqref{aboveformula}. Recall \eqref{defgtau} that $g_\tau=\sum_{\theta\in\Theta(\tau)}|f_\theta|^2$. We will show there is a local orthogonality for $\{g_\tau\}_{\tau\in\Tau(\om)}$ on $U$, that is, 
we claim the following estimate:
\begin{equation}\label{L4onU}
   \int_U |\sum_{\tau\in\Tau(\om)} (\Id^*_{\theta_\tau}-\Id^*_{\theta_{\tau,low}})^\vee*g_\tau |^2\lesssim \int \chi_U |f_{\Tau(\om)}|^4,
\end{equation}
where $\chi_U$ is the indicator function of $U$ with rapidly decaying tail.

Our first step is to rewrite $(\Id^*_{\theta_\tau}-\Id^*_{\theta_{\tau,low}})^\vee*g$. 
If $\tau\in\Tau(\om)$, then by definition $\theta_{\tau}\setminus \theta_{\tau,low}\subset \om\cup \om_{re}$. If we denote by $\wt\om$ the translation of $\om$ to the center, from a simple geometry we have $ \theta_\tau\subset 100\wt\om $. We may omit the constant $100$ and write it as $\theta_\tau\subset \wt\om$. For reader's convenience, we remark that $\theta_\tau$ is of dimensions $\de\times\de^{1/2}\times \ga$, and $\wt\om$ is of dimensions $\ga^{-1}\de\times\de^{1/2}\times\ga(=R^{-1}\ga\times R^{-1/2}\ga\times \ga)$. Like the definition of $\theta_{\tau,low}$, we define $\wt\om_{low}$ to be the portion of $\wt\om$ that centered at the origin of dimensions $R^{-1}\ga\times R^{-1/2}\ga\times K^{-1}\ga$  (actually $\wt\om$ and $
\wt\om_{low}$ are comparable since $K$ is a large constant). Next, we will use the cutoff replacing property (see Section \ref{replace}).
Noting that $\wh g_\tau\subset \theta_\tau$, we can choose two bump functions $\Id^*_{\wt\om}$, $\Id^*_{\wt\om_{low}}$ adapted to $\wt\om,  \wt\om_{low}$ respectively, so that $\Id^*_{\theta_\tau}=\Id^*_{\wt\om}, \Id^*_{\theta_{\tau,low}}=\Id^*_{\wt\om_{low}}$ on $\supp \wt g_\tau$. As a result, we have
$(\Id^*_{\theta_\tau}-\Id^*_{\theta_{\tau,low}})^\vee*g_\tau=(\Id^*_{\wt\om}-\Id^*_{\wt\om_{low}})^\vee*g_\tau$. In other words, one can rewrite the convolution as 
$$ (\Id^*_{\theta_\tau}-\Id^*_{\theta_{\tau,low}})^\vee*g_\tau=\vp_{\wt\om}^\vee*g_\tau, $$
for some $\vp_{\wt\om}$ adapted to $\wt\om$, by taking the advantage that $\supp \wh g_\tau\subset \theta_\tau$. Thus,
$$ |\vp_{\wt\om}^\vee*g_\tau|\lesssim \frac{1}{|\om^*|}\chi_{\om^*}*g_\tau, $$
where $\chi_{\om^*}$ is the indicator function of $\om^*$ with rapidly decaying tail.

Let us prove \eqref{L4onU}.
\begin{align}
    \int_U |\sum_{\tau\in\Tau(\om)} (\Id^*_{\theta_\tau}-\Id^*_{\theta_{\tau,low}})^\vee*g_\tau |^2\lesssim \int \Id_U \big(\frac{1}{|\om^*|}\chi_{\om^*}*\sum_{\tau\in\Tau(\om)}\sum_{\theta\in\Theta(\tau)} |f_\theta|^2\big)^2\\
    \lesssim \int (\frac{1}{|\om^*|}\chi_{\om^*}*\Id_U)\cdot \big(\sum_{\tau\in\Tau(\om)}\sum_{\theta\in\Theta(\tau)} |f_\theta|^2\big)^2\lesssim \int \chi_U \big(\sum_{\tau\in\Tau(\om)}\sum_{\theta\in\Theta(\tau)} |f_\theta|^2\big)^2.
\end{align}
In the last inequality, we used $\frac{1}{|\om^*|}\chi_{\om^*}*\Id_U\lesssim \chi_U$ since $\om^*$ is contained in the translation of $U$ to the origin.

Next, we will apply Lemma \ref{weightedL2}. Note that $\theta_\tau$ is contained in a translated copy of $\om$, so each $\theta$ is contained in a translated copy of $\om$. To indicate its relationship with $\theta$, we denote this translated copy by $\om_\theta$ so that $\theta\subset \om_\theta$. One easily sees that $\{\om_\theta\}_\theta$ are finitely overlapping. Also we can find a smooth bump function $\Id^*_{\om_\theta}$ adapted to $\om_\theta$ such that $\Id^*_{\om_\theta}=\Id^*_\theta$ on $\supp \wh f$, by taking the advantage that $\supp\wh f\subset N_{\de}(\Ga)$ and the cutoff replacing property. By duality, we choose a function $g$ with $\|g\|_2=1$, such that 
\begin{equation}\label{right}
    \bigg(\int \chi_U \big(\sum_{\tau\in\Tau(\om)}\sum_{\theta\in\Theta(\tau)} |f_\theta|^2\big)^2\bigg)^{1/2}=\int g \chi_U^{1/2} \sum_{\tau\in\Tau(\om)}\sum_{\theta\in\Theta(\tau)} |f_\theta|^2.
\end{equation}
Note that $f_\theta=(\Id^*_\theta \wh f)^\vee=(\Id^*_\theta \Id^*_{\Tau(\om)}\wh f)^\vee=(\Id^*_\theta \wh f_{\Tau(\om)})^\vee=(\Id^*_{\om_\theta} \wh f_{\Tau(\om)})^\vee$. Since $\{\om_\theta\}$ is a set of congruent rectangles,
by applying Lemma \ref{weightedL2} to the function $f_{\Tau(\om)}$, \eqref{right} is bounded by
\begin{align}
    \int \frac{1}{|\om^*|}\chi_{\om^*}*(g\chi_U^{1/2}) |f_{\Tau(\om)}|^2=\int \frac{\frac{1}{|\om^*|}\chi_{\om^*}*(g\chi_U^{1/2})}{\chi_U^{1/2}} |f_{\Tau(\om)}|^2\chi_U^{1/2}\\
    \le \bigg(\int \big|\frac{\frac{1}{|\om^*|}\chi_{\om^*}*(g\chi_U^{1/2})}{\chi_U^{1/2}}\big|^2\bigg)^{1/2}\bigg( \int \chi_U |f_{\Tau(\om)}|^4\bigg)^{1/2}.
\end{align}
To finish the estimate, we just note that 
\begin{align*}
    \int \big(\frac{\frac{1}{|\om^*|}\chi_{\om^*}*(g\chi_U^{1/2})}{\chi_U^{1/2}}\big)^2\lesssim \int \bigg(\frac{1}{|\om^*|}\chi_{\om^*} *(|g|^2\chi_U)\bigg)\cdot\chi_U^{-1}
    =\int |g|^2\chi_U\cdot\bigg(\frac{1}{|\om^*|}\chi_{\om^*} *\chi_U^{-1}\bigg)\\
    \lesssim \int |g|^2=1. 
\end{align*} 
Here in the last inequality, we use the fact that $\chi_U^{-1}\sim \frac{1}{|\om^*|}\chi_{\om^*}*\chi_U^{-1} $, since $\om^*$ is contained in a translation of $U$.
We finish the proof of \eqref{L4onU}.

Plugging back to \eqref{comeback} and by Lemma \ref{kakeyalem}, we have
\begin{equation}
    \int_{\R^3}|\sum_\tau g_{\tau,high}|^4\le (C_{\e'} K^{O(1)}(\ga^{-2}\de)^{-\e'})^4 \sum_{R^{-1/2}\le s\le 1}\sum_{\om_s\in\Om_s}\sum_{U\parallel U_{\om_s}}|U|^{-1}\big(\int \chi_U \sum_{\om\subset\om_s}|f_{\Tau(\om)}|^4 \big)^2.
\end{equation}

By  Lemma \ref{localinterp}, we have
\begin{equation}
    \int \chi_U \sum_{\om\subset\om_s}|f_{\Tau(\om)}|^4\lesssim \int \chi_U |f_{\Tau(\om_s)}|^4.
\end{equation}
Here $f_{\Tau(\om_s)}$ has the similar definition as $f_{\Tau(\om)}$ does.
So, we have
\begin{align}
    \int_{\R^3}|\sum_\tau g_{\tau,high}|^4&\le (C_{\e'} K^{O(1)}(\ga^{-2}\de)^{-\e'})^4 \sum_{R^{-1/2}\le s\le 1}\sum_{\om_s\in\Om_s}\sum_{U\parallel U_{\om_s}}|U|^{-1}\big(\int \chi_U |f_{\om_s}|^4 \big)^2\\
    &\lesssim (C_{\e'} K^{O(1)}(\ga^{-2}\de)^{-\e'})^4 \sum_{R^{-1/2}\le s\le 1}\sum_{\om_s\in\Om_s}\sum_{U\parallel U_{\om_s}}\int \chi_U |f_{\om_s}|^8 \\
    &\sim (C_{\e'} K^{O(1)}(\ga^{-2}\de)^{-\e'})^4 \sum_{R^{-1/2}\le s\le 1}\sum_{\om_s\in\Om_s}\int_{\R^3}|f_{\om_s}|^8\\
    &\lesssim (C_{\e'} K^{O(1)}(\ga^{-2}\de)^{-\e'})^4\log\de^{-1} \int_{\R^3}|f|^8,
\end{align}
where the last inequality is by Lemma \ref{interpolation}.

Combining the estimate for the low term \eqref{lowest}, we just need to show
\begin{equation}
    C_\e(\ga K^{-1}\de^{-1})^\e+ C_{\e'} K^{O(1)}(\ga^2\de^{-1})^{\e'} \log\de^{-1}\le \frac{1}{C} C_\e (\ga\de^{-1})^\e
\end{equation}
in order to close the induction.

By choosing $K$ large enough and $\e'\ll\e$ (for example $\e'=\e^2$), we close the induction.
\end{proof}


\appendix
\section{Examples}
In the appendix, we give some examples.
Before doing that, we discuss a property of wave packet which will be used to construct examples. Our argument here is heuristic, but is not hard to be made rigorous.

Let $R$ be a rectangle in the frequency space $\R^n_\xi$. After rotation, we may write it as $R=c_R+\prod_{i=1}^n [-a_i,a_i]$, where $c_R$ is the center of $R$. Correspondingly, its dual rectangle is given by $R^*=\prod_{i=1}^n [-1/a_i,1/a_i]\subset \R^n_x$. 
One heuristic we will use in the rest of Appendix is:
\begin{equation}\label{heuristic0}
    (\frac{1}{|R|}\Id_R)^\vee(x)\approx e^{2\pi i x\cdot c_R}  \Id_{R^*}(x).
\end{equation}
By adding a phase to $\Id_R$, we also have
\begin{equation}\label{heuristic}
    (e^{-2\pi i x_0\cdot \xi}\frac{1}{|R|}\Id_R(\xi))^\vee(x)\approx e^{2\pi i (x-x_0)\cdot c_R}  \Id_{x_0+R^*}(x).
\end{equation}
In other words, there is a function whose support lie in $R$ and whose inverse Fourier transform, after taking absolute value, is roughly the indicator function of a translation of $R^*$.

\begin{remark}
\rm
$e^{2\pi i (x-x_0)\cdot c_R}  \Id_{x_0+R^*}(x)$ is referred to as \textit{a wave packet} at $x_0+R^*$.
\end{remark}

Let us discuss a trick called  \textit{wave packet dilation}. Given $\theta'=c+\prod_{i=1}^n [-a_i,a_i]$, we specify a direction, for example,  $\boldsymbol e_n$ (equivalently, the $\xi_n$-direction), and then let $\theta=c+\prod_{i=1}^{n-1} [-a_i,a_i]\times [0,a_n]$ be the upper half of $\theta'$. We see that the dual $\theta^{*}$ is the $2$-dilation of $\theta'^*$ along $\boldsymbol e_n$.  Write $\theta^{*}=Dil_{2,\boldsymbol e_n}\theta'^*$. When the direction of the dilation is clear (usually the direction is along the longest side or the second longest side), we just denote it by $\theta^{*}=Dil_{2}\theta'^*$. When we look for the examples of the map $T:f\mapsto (\Id_\theta \wh f)^\vee$, the auxiliary rectangle $\theta'$ will be very helpful. If we choose the test function $f=\frac{1}{|\theta|}(e^{-2\pi i x_0\cdot\xi}\Id_{\theta'})^\vee$, then
$$ |f|\approx \Id_T,\ \  |Tf|\approx \Id_{Dil_2 T}, $$
where $T=x_0+\theta'^*$ is a translation of $\theta^*$. It means that after the action of $T$ on the single wave packet $f$, the resulting new wave packet $Tf$ is two times longer.

By using this idea, we can quickly give the sharp example for Bochner-Riesz conjecture. First, let us recall the Bochner-Riesz conjecture. Let $N_{R^{-1}}(\ZS^{n-1})$ be the $\de$-neighborhood of the unit sphere in $\R^n$. Let $\Id^*_{N_{R^{-1}}(\ZS^{n-1})}$ be a smooth bump function adapted to $N_{R^{-1}}(\ZS^{n-1})$. Define the operator $Sf:=(\Id^*_{N_{R^{-1}}(\ZS^{n-1})} \wh f)^\vee$. We are interested in the following estimate
\begin{equation}\label{BR}
    \|Sf\|_{L^p(\R^n)}\lesssim R^{C_{n,p}} \|f\|_{L^p(\R^n)}.
\end{equation}
The conjecture is: \eqref{BR} holds for $p>\frac{2n}{n-1}$ and $C_{n,p}>\frac{n-1}{2}-\frac{n}{p}$.
We construct an example for \eqref{BR}. First, we write $\Id^*_{N_{R^{-1}}(\ZS^{n-1})}=\sum_\theta \Id^*_\theta$. Where $\{\theta\}$ is a set of $R^{-1/2}\times\cdots\times R^{-1/2}\times R^{-1}$-slabs that cover $N_{R^{-1}}(\ZS^{n-1})$.
For each $\theta$, we define $\theta'$ to be a $R^{-1/2}\times\cdots\times R^{-1/2}\times 100R^{-1}$-slab that contains $\theta$. Actually, $\theta'$ is the $100$-dilation of $\theta$ along the normal direction of $\theta$. Heuristically, we may assume $\{\theta'\}$ are disjoint. For each $\theta'$, we choose a function $f_{\theta'}$ such that $\supp \wh f_{\theta'}\subset \theta'$ and $|f_{\theta'}|\approx \Id_{T_{\theta'}}$, $|(\Id^*_\theta \wh f_{\theta'})^\vee|\approx \Id_{Dil_{100} T_{\theta'}}$, where $T_{\theta'}$ is a $R^{1/2}\times\cdots\times R^{1/2}\times 100^{-1}R$-tube dual to $\theta'$.
Now we just choose these tubes so that $\{T_{\theta'}\}$ are disjoint and $\{Dil_{100}T_{\theta'}\}$ intersect at the origin.

We choose our example $f=\sum_\theta a_\theta f_{\theta'}$ where $a_\theta\in \mathbb C$ are to be determined.
Since $\{T_{\theta'}\}$ are disjoint, we have $\|f\|_p\approx \|\sum_\theta \Id_{T_\theta'}\|_p\sim R^{\frac{n}{p}}$.
Since $\{\theta'\}$ are disjoint, we have $Sf=\sum_\theta (\Id^*_\theta \wh f)^\vee=\sum_\theta (a_\theta \Id^*_\theta \wh f_{\theta'})^\vee$. We can make it have a constructive interference at the unit ball $B_1(0)$ by properly choosing $a_\theta$. As a result, $\|Sf\|_p\gtrsim \#\theta\sim R^{\frac{n-1}{2}}$. Plugging into \eqref{BR} verifies $C_{n,p}\ge \frac{n-1}{2}-\frac{n}{p}$.

\begin{figure}

\begin{tikzpicture}
 
\draw[color=blue, thick] (-10,8.66) arc (120:60:10);
\draw[color=blue, thick] (-9.5,7.794) arc (120:60:9);
\draw[color=blue, thick] (-9,6.96) arc (120:60:8);
\draw[color=blue, thick] (-8.5,6.06) arc (120:60:7);

\draw[color=blue,rotate around={120:(-5,0)}] (2,0) -- (3,0);
\draw[color=blue,rotate around={110:(-5,0)}] (2,0) -- (3,0);
\draw[color=blue,rotate around={100:(-5,0)}] (2,0) -- (3,0);
\draw[color=blue,rotate around={90:(-5,0)}] (2,0) -- (3,0);
\draw[color=blue,rotate around={80:(-5,0)}] (2,0) -- (3,0);
\draw[color=blue,rotate around={70:(-5,0)}] (2,0) -- (3,0);
\draw[color=blue,rotate around={60:(-5,0)}] (2,0) -- (3,0);

\draw[color=blue,rotate around={120:(-5,0)}] (4,0) -- (5,0);
\draw[color=blue,rotate around={110:(-5,0)}] (4,0) -- (5,0);
\draw[color=blue,rotate around={100:(-5,0)}] (4,0) -- (5,0);
\draw[color=blue,rotate around={90:(-5,0)}] (4,0) -- (5,0);
\draw[color=blue,rotate around={80:(-5,0)}] (4,0) -- (5,0);
\draw[color=blue,rotate around={70:(-5,0)}] (4,0) -- (5,0);
\draw[color=blue,rotate around={60:(-5,0)}] (4,0) -- (5,0);

\draw[color=red,thick,rotate around={115:(-5,0)}] (4.9,-.4) rectangle (5.1,.6);
\draw[color=red,thick,rotate around={105:(-5,0)}] (4.9,-.4) rectangle (5.1,.6);
\draw[color=red,thick,rotate around={95:(-5,0)}] (4.9,-.4) rectangle (5.1,.6);
\draw[color=red,thick,rotate around={85:(-5,0)}] (4.9,-.4) rectangle (5.1,.6);
\draw[color=red,thick,rotate around={75:(-5,0)}] (4.9,-.4) rectangle (5.1,.6);
\draw[color=red,thick,rotate around={65:(-5,0)}] (4.9,-.4) rectangle (5.1,.6);


\draw[color=red,thick,rotate around={115:(-5,0)}] (2.9,-.44) rectangle (3.1,.56);
\draw[color=red,thick,rotate around={105:(-5,0)}] (2.9,-.44) rectangle (3.1,.56);
\draw[color=red,thick,rotate around={95:(-5,0)}] (2.9,-.44) rectangle (3.1,.56);
\draw[color=red,thick,rotate around={85:(-5,0)}] (2.9,-.44) rectangle (3.1,.56);
\draw[color=red,thick,rotate around={75:(-5,0)}] (2.9,-.44) rectangle (3.1,.56);
\draw[color=red,thick,rotate around={65:(-5,0)}] (2.9,-.44) rectangle (3.1,.56);

\draw[->, thick] (-5,6) -- (-5,5);

\draw[color=blue, ] (-7.5,3.5) -- (-7.5,2.5);
\draw[color=blue, ] (-6.5,3.5) -- (-6.5,2.5);
\draw[color=blue, ] (-7.5,2.5) -- (-6.5,2.5);
\draw[color=blue, thick] (-7.5,3.5) -- (-6.5,3.5);
\draw[color=red, thick] (-7.4,3.6) rectangle (-6.6,3.4);

\node at (-5,3) {where};
\node[scale=1.3] at (0,6) {$\subset \ZS^{n-1}$};

\draw[color=blue, ] (-3.5,4) rectangle (-2.5,3);

\node at (-1,3.5) {is a $\de$-cube};

\draw[color=red, thick] (-3.4,2.2) rectangle (-2.6,2.4);

\node at (-0,2.3) {is a $\de\times\cdots\times\de\times\de^2$-slab};

\end{tikzpicture}
\caption{Fourier support}
\label{sharpcutexample}
\end{figure}
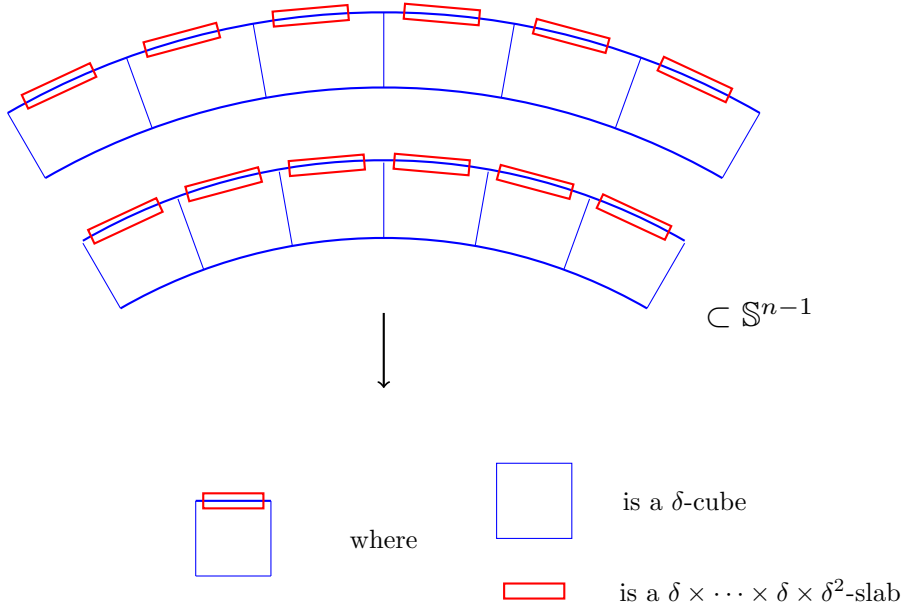

\subsection{} We show that if we remove the ``one-dimensional'' condition in Theorem \ref{main-thm-2}, then \eqref{main-esti-2} is no longer true. See in Figure \ref{sharpcutexample} where we plot our $\{\De_i\}$ (blue $\de$-cube) on the sphere $\ZS^{n-1}$ whose normal directions will be specified latter. There are also some red slabs of dimensions $\de\times\cdots\times\de\times \de^2$, each of which is \textit{attached} to only one $\De_j$, i.e., one half of the red slab lies in $\De_j$ and another half is outside $\De_j$. We denote the slab attached to $\De_j$ by $R_j$, and the set of them by $\{R_j\}$. 
We consider the corresponding regions in $\R^n$. Define
$$ \wt \De_j=\{ \xi\in\R^n:\xi/|\xi|\in\De_j, 1-\de\le|\xi|\le 1 \}, $$
$$ \wt R_j=\{ \xi\in\R^n:\xi/|\xi|\in R_j, 1-\de\le|\xi|\le 1 \}. $$
We see that $\{\wt \De_j\}$ is a set of $\de$-cubes in $\R^n$, and $\{\wt R_j\}$ is a set of $\de\times\cdots\times\de\times\de^2$-slabs. Denote the normal direction of $\wt R_j$ by $\vec n_j$. We may choose $\{\De_j\}$ and $\{R_j\}$ carefully so that $\{\vec n_j\}$ form a $\de$-dense subset of $\ZS^{n-1}$.

Now, for each $\De_j$, we choose a function $f_j$ so that $\wh f_j\subset \wt R_j$ and $$|f_j|\approx \Id_{T_j},\ \ |S_j f_j|\approx |(\Id_{\wt \De_j}\wh f_j)^\vee|\approx \Id_{Dil_{2}T_j},$$
where the dilation $Dil_{2}T_j$ is along $\vec n_j$, and $T_j$ is a $\de^{-1}\times\cdots\times\de^{-1}\times\de^{-2}$-tube dual to $\wt R_j$.
We choose the tubes $\{T_j\}$ so that they are disjoint and their dilations $\{Dil_{100}T_j\}$ intersect the origin. See in Figure \ref{sharpcutexample2} where the blue tubes are the $T_j$'s and they intersect $B_{\de^{-1}}$ at the origin. 

\begin{figure}
\begin{tikzpicture}
\begin{scope}
 [xshift=-18.5, yshift=10]
  \draw[fill=red] (0,0,0) -- (0,0,1.5) -- (0,1.5,1.5) -- (0,1.5,0) -- cycle;
  \draw[fill=red] (0,0,0) -- (1.5,0,0) -- (1.5,1.5,0) -- (0,1.5,0) -- cycle;
  \draw[fill=red] (0,1.5,0) -- (1.5,1.5,0) -- (1.5,1.5,1.5) -- (0,1.5,1.5) -- cycle;
  \draw[fill=red] (1.5,0,0) -- (1.5,0,1.5) -- (1.5,1.5,1.5) -- (1.5,1.5,0) -- cycle;
  \draw[fill=red] (0,0,1.5) -- (1.5,0,1.5) -- (1.5,1.5,1.5) -- (0,1.5,1.5) -- cycle;
  
 \end{scope}

 \begin{scope}
   [x={(4cm,0cm)},
    y={({cos(90)*1cm},{sin(90)*1cm})},
    z={({cos(70)*0.8cm},{sin(70)*0.8cm})},line join=round,fill opacity=0.5,blue]
  \draw[dotted] (0,0,0) -- (0,0,1) -- (0,1,1) -- (0,1,0) -- cycle;
  \draw[dotted] (0,0,0) -- (1,0,0) -- (1,1,0) -- (0,1,0) -- cycle;
  \draw[dotted] (0,1,0) -- (1,1,0) -- (1,1,1) -- (0,1,1) -- cycle;
  \draw[dotted] (1,0,0) -- (1,0,1) -- (1,1,1) -- (1,1,0) -- cycle;
  \draw[dotted] (0,0,1) -- (1,0,1) -- (1,1,1) -- (0,1,1) -- cycle;
  \draw[dotted] (0,0,0) -- (1,0,0) -- (1,0,1) -- (0,0,1) -- cycle;
  
  \draw[] (1,0,0) -- (1,0,1) -- (1,1,1) -- (1,1,0) -- cycle;
  \draw[] (1,0,0) -- (2,0,0) -- (2,1,0) -- (1,1,0) -- cycle;
  \draw[] (1,1,0) -- (2,1,0) -- (2,1,1) -- (1,1,1) -- cycle;
  \draw[] (2,0,0) -- (2,0,1) -- (2,1,1) -- (2,1,0) -- cycle;
  \draw[] (1,0,1) -- (2,0,1) -- (2,1,1) -- (1,1,1) -- cycle;
  \draw[] (1,0,0) -- (2,0,0) -- (2,0,1) -- (1,0,1) -- cycle;
 \end{scope}
 \begin{scope}
   [rotate=20,x={(4cm,0cm)},
    y={({cos(90)*.8cm},{sin(90)*.8cm})},
    z={({cos(70)*.8cm},{sin(70)*.8cm})},line join=round,fill opacity=0.5,blue]
  \draw[dotted] (0,0,0) -- (0,0,1) -- (0,1,1) -- (0,1,0) -- cycle;
  \draw[dotted] (0,0,0) -- (1,0,0) -- (1,1,0) -- (0,1,0) -- cycle;
  \draw[dotted] (0,1,0) -- (1,1,0) -- (1,1,1) -- (0,1,1) -- cycle;
  \draw[dotted] (1,0,0) -- (1,0,1) -- (1,1,1) -- (1,1,0) -- cycle;
  \draw[dotted] (0,0,1) -- (1,0,1) -- (1,1,1) -- (0,1,1) -- cycle;
  \draw[dotted] (0,0,0) -- (1,0,0) -- (1,0,1) -- (0,0,1) -- cycle;
  
  \draw[] (1,0,0) -- (1,0,1) -- (1,1,1) -- (1,1,0) -- cycle;
  \draw[] (1,0,0) -- (2,0,0) -- (2,1,0) -- (1,1,0) -- cycle;
  \draw[] (1,1,0) -- (2,1,0) -- (2,1,1) -- (1,1,1) -- cycle;
  \draw[] (2,0,0) -- (2,0,1) -- (2,1,1) -- (2,1,0) -- cycle;
  \draw[] (1,0,1) -- (2,0,1) -- (2,1,1) -- (1,1,1) -- cycle;
  \draw[] (1,0,0) -- (2,0,0) -- (2,0,1) -- (1,0,1) -- cycle;
 \end{scope}
 \begin{scope}
   [rotate=40,x={(4cm,0cm)},
    y={({cos(90)*.8cm},{sin(90)*.8cm})},
    z={({cos(70)*.8cm},{sin(70)*.8cm})},line join=round,fill opacity=0.5,blue]
  \draw[dotted] (0,0,0) -- (0,0,1) -- (0,1,1) -- (0,1,0) -- cycle;
  \draw[dotted] (0,0,0) -- (1,0,0) -- (1,1,0) -- (0,1,0) -- cycle;
  \draw[dotted] (0,1,0) -- (1,1,0) -- (1,1,1) -- (0,1,1) -- cycle;
  \draw[dotted] (1,0,0) -- (1,0,1) -- (1,1,1) -- (1,1,0) -- cycle;
  \draw[dotted] (0,0,1) -- (1,0,1) -- (1,1,1) -- (0,1,1) -- cycle;
  \draw[dotted] (0,0,0) -- (1,0,0) -- (1,0,1) -- (0,0,1) -- cycle;
  
  \draw[] (1,0,0) -- (1,0,1) -- (1,1,1) -- (1,1,0) -- cycle;
  \draw[] (1,0,0) -- (2,0,0) -- (2,1,0) -- (1,1,0) -- cycle;
  \draw[] (1,1,0) -- (2,1,0) -- (2,1,1) -- (1,1,1) -- cycle;
  \draw[] (2,0,0) -- (2,0,1) -- (2,1,1) -- (2,1,0) -- cycle;
  \draw[] (1,0,1) -- (2,0,1) -- (2,1,1) -- (1,1,1) -- cycle;
  \draw[] (1,0,0) -- (2,0,0) -- (2,0,1) -- (1,0,1) -- cycle;
 \end{scope}
 \begin{scope}
   [rotate=60,x={(4cm,0cm)},
    y={({cos(90)*.8cm},{sin(90)*.8cm})},
    z={({cos(70)*.8cm},{sin(70)*.8cm})},line join=round,fill opacity=0.5,blue]
  \draw[dotted] (0,0,0) -- (0,0,1) -- (0,1,1) -- (0,1,0) -- cycle;
  \draw[dotted] (0,0,0) -- (1,0,0) -- (1,1,0) -- (0,1,0) -- cycle;
  \draw[dotted] (0,1,0) -- (1,1,0) -- (1,1,1) -- (0,1,1) -- cycle;
  \draw[dotted] (1,0,0) -- (1,0,1) -- (1,1,1) -- (1,1,0) -- cycle;
  \draw[dotted] (0,0,1) -- (1,0,1) -- (1,1,1) -- (0,1,1) -- cycle;
  \draw[dotted] (0,0,0) -- (1,0,0) -- (1,0,1) -- (0,0,1) -- cycle;
  
  \draw[] (1,0,0) -- (1,0,1) -- (1,1,1) -- (1,1,0) -- cycle;
  \draw[] (1,0,0) -- (2,0,0) -- (2,1,0) -- (1,1,0) -- cycle;
  \draw[] (1,1,0) -- (2,1,0) -- (2,1,1) -- (1,1,1) -- cycle;
  \draw[] (2,0,0) -- (2,0,1) -- (2,1,1) -- (2,1,0) -- cycle;
  \draw[] (1,0,1) -- (2,0,1) -- (2,1,1) -- (1,1,1) -- cycle;
  \draw[] (1,0,0) -- (2,0,0) -- (2,0,1) -- (1,0,1) -- cycle;
 \end{scope}
 \begin{scope}
   [rotate=80,x={(4cm,0cm)},
    y={({cos(90)*.8cm},{sin(90)*.8cm})},
    z={({cos(70)*.8cm},{sin(70)*.8cm})},line join=round,fill opacity=0.5,blue]
  \draw[dotted] (0,0,0) -- (0,0,1) -- (0,1,1) -- (0,1,0) -- cycle;
  \draw[dotted] (0,0,0) -- (1,0,0) -- (1,1,0) -- (0,1,0) -- cycle;
  \draw[dotted] (0,1,0) -- (1,1,0) -- (1,1,1) -- (0,1,1) -- cycle;
  \draw[dotted] (1,0,0) -- (1,0,1) -- (1,1,1) -- (1,1,0) -- cycle;
  \draw[dotted] (0,0,1) -- (1,0,1) -- (1,1,1) -- (0,1,1) -- cycle;
  \draw[dotted] (0,0,0) -- (1,0,0) -- (1,0,1) -- (0,0,1) -- cycle;
  
  \draw[] (1,0,0) -- (1,0,1) -- (1,1,1) -- (1,1,0) -- cycle;
  \draw[] (1,0,0) -- (2,0,0) -- (2,1,0) -- (1,1,0) -- cycle;
  \draw[] (1,1,0) -- (2,1,0) -- (2,1,1) -- (1,1,1) -- cycle;
  \draw[] (2,0,0) -- (2,0,1) -- (2,1,1) -- (2,1,0) -- cycle;
  \draw[] (1,0,1) -- (2,0,1) -- (2,1,1) -- (1,1,1) -- cycle;
  \draw[] (1,0,0) -- (2,0,0) -- (2,0,1) -- (1,0,1) -- cycle;
 \end{scope}
 \node at (-1,-0.8) {$B_{\de^{-1}}$};
 \node at (9,1) {$T_j$'s};
\end{tikzpicture}
\caption{Concentration of tubes}
\label{sharpcutexample2}
\end{figure}
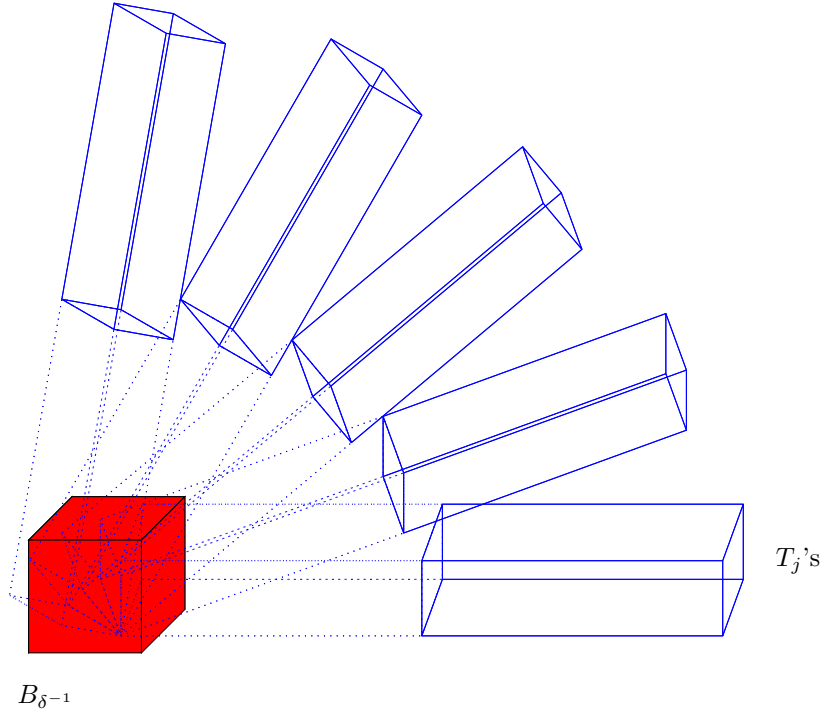

We choose the test function $f=\sum_j f_j$ and plug into
\begin{equation}\label{test}
    \|(\sum_j |S_j f|^2)^{1/2}\|_p\le C \|f\|_p.
\end{equation}
Since $\{T_j\}$ are disjoint, we have $\|f\|_p\approx \|\sum_j \Id_{T_j}\|_p\sim \de^{-\frac{2n}{p}}$. Since $\{\wt \De_j\}$ are disjoint, we have
$|S_j f|=|S_j f_j|\approx \Id_{Dil_2 T_j}$.
We have 
$$\|(\sum_j |S_j f|^2)^{1/2}\|_p\gtrsim \|(\sum_j \Id_{T_j})^{1/2}\|_{L^p(B_{\de^{-1}})}\gtrsim \#\{\De_j\}^{1/2}\de^{-\frac{n}{p}}\sim \de^{-\frac{n-1}{2}-\frac{n}{p}}.$$
We see the constant $C$ in \eqref{test} should be greater than $\de^{-\frac{n-1}{2}+\frac{n}{p}}$ for which the threshold is $p=\frac{2n}{n-1}$. 

\subsection{}
Finally, we give the sharp example for Theorem \ref{L8thm}. For a $\de\times\de^{1/2}\times 1$-slab $\tau$ contained in $N_\de(\Ga)$, we use $\boldsymbol c(\tau)$, $\boldsymbol n(\tau)$ and $\boldsymbol t(\tau)$ to denote the light direction, normal direction and tangent direction of $\tau$. More precisely, if $\xi\in\tau\cap\Ga\cap \{\xi_3=1\}$, then we roughly have $\boldsymbol c(\tau)=(\xi_1,\xi_2,1)$, $\boldsymbol n(\tau)=(\xi_1,\xi_2,-1)$ and $\boldsymbol t(\tau)=(-\xi_2,\xi_1,0)$. 
In the condition of the Theorem \ref{L8thm}, we assumed $\wh f\subset N_\de(\Ga)$, so we cannot dilate the wave packet in the $\boldsymbol n(\tau)$-direction, which is the longest direction of $\tau^*$. However, the ``wave packet dilation" trick still works because we can dilate in the second longest direction $\boldsymbol t(\tau)$.

Let $\{\tau\}$ be $\de\times\de^{1/2}\times 1$-slabs contained in $N_\de(\Ga)$ such that $\{100\tau\}$ are disjoint. As what we did in the previous example, we can choose $f=\sum_\tau f_{100\tau}$, such that $\wh f_{100\tau}\subset 100\tau$ and 
$$ |f_{100\tau}|\approx \Id_{P_\tau},\ \ |(\Id^*_\tau \wh f_{100\tau})^\vee|\approx\Id_{Dil_{100}P_\tau}, $$
where the dilation $Dil_{100}P_\tau$ is along $\boldsymbol t(\tau)$, and $P_\tau$ is a $1\times \de^{1/2}\times\de$-plank dual to $\tau$.
Now, we carefully choose $\{P_\tau\}$ so that $\{P_\tau\}$ are disjoint but $\{Dil_{100}P_\tau\}$ intersect the unit ball at the origin. See in Figure \ref{L8example}. We arrange $\{P_\tau\}$ into the $\de^{-1/2}$-neighborhood of the hyperboloid $\{x\in\R^3: x_1^2+x_2^2-x_3^2=\de^{-1}\}$. Each $P_\tau$ intersect $\{x_3=0\}$ at a $1\times \de^{1/2}$-rectangle lying in $\{(x_1,x_2):\de^{-1/2}\le \sqrt{x_1^2+x_2^2}\le 2\de^{-1/2}\}$ and pointing to the origin.

The total measure of all the planks is $\de^{-2}$, so $\|f\|_p\approx \de^{-\frac{2}{p}}$. On the other hand, 
$$ \|(\sum_\tau |(\Id^*_\tau \wh f_{100\tau})^\vee|^2)^{1/2}\|_p\gtrsim \|(\sum_\tau \Id_{Dil_{100}P_\tau})^{1/2}\|_{L^p(B_1(0))}\sim \de^{-1/4}. $$
Plugging into 
$$ \|(\sum_\tau |(\Id^*_\tau \wh f_{100\tau})^\vee|^2)^{1/2}\|_p\le C\|f\|_p, $$
we get $C\ge \de^{-\frac{1}{4}+\frac{2}{p}}$, yielding that $p=8$ is the critical exponent.

\begin{figure}
\begin{tikzpicture}

 \begin{scope}
   [rotate=10,shift={(-.4,-3.3)},x={(.1cm,0cm)},
    y={({cos(110)*1.5cm},{sin(110)*1.5cm})},
    z={({cos(60)*5cm},{sin(60)*5cm})},line join=round,fill opacity=0.5,black]
  \draw[fill=red!40] (0,0,0) -- (0,0,1) -- (0,1,1) -- (0,1,0) -- cycle;
  \draw[fill=red!40] (0,0,0) -- (1,0,0) -- (1,1,0) -- (0,1,0) -- cycle;
  \draw[fill=red!40] (0,1,0) -- (1,1,0) -- (1,1,1) -- (0,1,1) -- cycle;
  \draw[fill=red!40] (1,0,0) -- (1,0,1) -- (1,1,1) -- (1,1,0) -- cycle;
  \draw[fill=red!40] (0,0,1) -- (1,0,1) -- (1,1,1) -- (0,1,1) -- cycle;
  \draw[fill=red!40
  ] (0,0,0) -- (1,0,0) -- (1,0,1) -- (0,0,1) -- cycle;
 \end{scope}
 
 \begin{scope}
   [shift={(-1.3,-3.35)},x={(.1cm,0cm)},
    y={({cos(90)*1.2cm},{sin(90)*1.2cm})},
    z={({cos(60)*5.3cm},{sin(60)*5.3cm})},line join=round,fill opacity=0.5,black]
  \draw[fill=blue!40] (0,0,0) -- (0,0,1) -- (0,1,1) -- (0,1,0) -- cycle;
  \draw[fill=blue!40] (0,0,0) -- (1,0,0) -- (1,1,0) -- (0,1,0) -- cycle;
  \draw[fill=blue!40] (0,1,0) -- (1,1,0) -- (1,1,1) -- (0,1,1) -- cycle;
  \draw[fill=blue!40] (1,0,0) -- (1,0,1) -- (1,1,1) -- (1,1,0) -- cycle;
  \draw[fill=blue!40] (0,0,1) -- (1,0,1) -- (1,1,1) -- (0,1,1) -- cycle;
  \draw[fill=blue!40] (0,0,0) -- (1,0,0) -- (1,0,1) -- (0,0,1) -- cycle;
 \end{scope}
 
 \fill (-.11,.6) circle (.5mm);
\draw[] (-.11,0.6) ellipse (1cm and .3cm);
\draw[] (-.11,0.6) ellipse (3cm and 1cm);
\draw[dotted] (.13,.35) -- (.39,-.47);
\draw[dotted] (.36,.35) -- (1.3,-.3);
\node at (3,-2) {$1\times \frac{1}{100}\de^{-1/2}\times\de^{-1}$-planks};
\end{tikzpicture}

\caption{Planks arranged along a hyperboloid}
\label{L8example}

\end{figure}
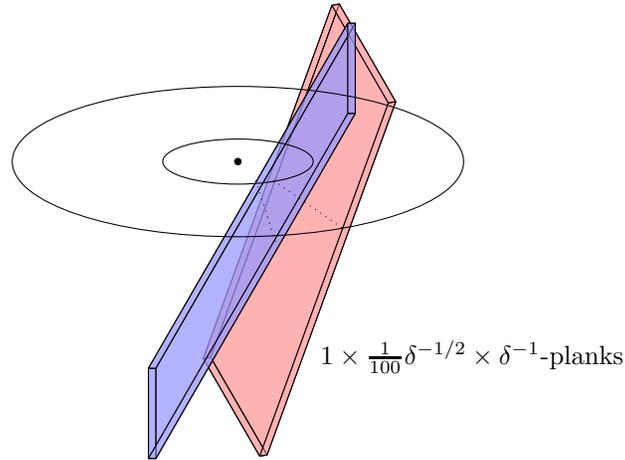

\newpage
\bibliographystyle{alpha}
\bibliography{bibli}

\end{document}